\documentclass[reqno]{amsart}

\usepackage[latin1]{inputenc}
\usepackage[english]{babel}
\usepackage{amsmath,amsfonts,amssymb,amsthm,fancyhdr}
\usepackage{bm}
\usepackage{graphicx}
\usepackage{color}
\usepackage{enumitem}
\usepackage[longnamesfirst]{natbib}

\theoremstyle{plain}

\newtheorem{theorem}{Theorem}[section]
\newtheorem{corollary}[theorem]{Corollary}
\newtheorem{definition}[theorem]{Definition}

\newtheorem{lemma}[theorem]{Lemma}

\newtheorem{proposition}[theorem]{Proposition}
\newtheorem{remark}[theorem]{Remark}

\newtheorem{assumption}[theorem]{Assumption}
\numberwithin{equation}{section}

\newcommand{\re}{{\mathbb R}}    
\newcommand{\cn}{{\mathbb C}}    

\newcommand{\supp}{\operatorname{supp}}

\newcommand{\im}{\operatorname{i}}

\newcommand{\tr}{\operatorname{tr}}
\newcommand{\rk}{\operatorname{rank}}

\newcommand{\relint}{\operatorname{relint}}
\newcommand{\Texpl}{\operatorname{T_{expl}}}
\newcommand{\Tjump}{\operatorname{T_{jump}}}
\newcommand{\Omegaexpl}{\operatorname{\Omega_{expl}}}
\newcommand{\aff}{\operatorname{aff}}

\renewcommand{\Im}{\mathrm{Im}}
\renewcommand{\Re}{\mathrm{Re}}

\begin{document}

\title[Path properties and regularity of affine processes]
{Path properties and regularity of affine processes on general state spaces}

\author{Christa Cuchiero and Josef Teichmann}
\address{University of Vienna, Faculty of Mathematics, Nordbergstrasse 15, A-1090 Wien, Austria
\newline ETH Z\"urich, Departement Mathematik, R\"amistrasse 101, 8092 Z\"urich, Switzerland}
\email{christa.cuchiero@univie.ac.at, josef.teichmann@math.ethz.ch}

\begin{abstract}
We provide a new proof for regularity of affine processes on general state spaces by methods from the theory of Markovian semimartingales. On the way to this result we also show that the definition of an affine process, namely as stochastically continuous time-homogeneous Markov process with exponential affine Fourier-Laplace transform, already implies the existence of a c\`adl\`ag version. This was one of the last open issues in the
fundaments of affine processes.
\end{abstract}

\thanks{We thank Georg Grafendorfer and Enno Veerman for discussions and helpful comments. \\Both authors gratefully acknowledge the financial support by the ETH Foundation.}
\keywords{affine processes, path properties, regularity, Markov semimartingales.  \textit{MSC 2000:} Primary: 60J25; Secondary: 60G17}

\maketitle

\section{Introduction}

In the last decades affine processes have been of great interest in mathematical finance to model phenomena like stochastic volatility, stochastic interest rates, heavy tails, credit default, etc.~Pars pro toto we mention here the one-dimensional short-rate model of~\citet{cir}, the stochastic volatility model of~\citet{heston} and the credit risk model of~\citet{lando}. In order to accommodate the more and more complex structures in finance, these simple models have progressively been extended to higher dimensional affine jump diffusions with values in the so-called canonical state space  $\re^m_+ \times \re^{n-m}$, or in the cone of positive semidefinite matrices, see, e.g.,~\citet{daising, dk, dps} for affine models on the canonical state space and~\citet{fonsecaetal1,leippoldtrojani,gourierouxsufana,bur_cie_tro_07} for  multivariate stochastic volatility and interest rate models based on matrix-valued affine processes. 

Axiomatically speaking affine processes are stochastically continuous Markov processes on some state space $ D \subseteq V$, where $V$ is a finite-dimensional Euclidean vector space with scalar product $\langle \cdot, \cdot \rangle$, such that the Fourier-Laplace transform is of exponential affine form in the initial values. More precisely, this means that there exist functions $\Phi$ and $\psi$ such that
\begin{align*}
\mathbb{E}_x\left[e^{\langle u, X_t\rangle} \right]=\Phi(t,u)e^{\langle \psi(t,u),x\rangle},
\end{align*}
for all $(t,x)\in \re_+ \times D$ and $u \in V + \im V$, for which $x \mapsto  e^{\langle u, x\rangle}$ is a bounded function on $D$.
From this definition neither the Feller property, nor the existence of a c\`adl\`ag version, nor differentiability of the Fourier-Laplace transform with respect to time, a concept called \emph{regularity} (see~\citet[Definition 2.5]{dfs}), are immediate. This paper provides a positive answer to the latter two questions, while the Feller property is still an open issue on general state spaces, but can probably be established by building on the results of the present article. 

The reasons for the strong interest in affine processes are twofold: first, affine processes are a rich and flexible class of Markov processes containing L\'evy processes, Ornstein-Uhlenbeck processes, squared Bessel processes and aggregates of them. Second, affine processes are analytically tractable in the sense that the Fourier-Laplace transform, which is a solution of the backward Kolmogorov equation, a PIDE with affine coefficients, can be calculated by solving a system of ODEs for $ \Phi $ and $ \psi $, the so-called generalized Riccati equations. Having the Fourier-Laplace transform at hand then means that real-time-calibration is at reach from a numerical point of view. However in order to show that the functions $\Phi$ and $\psi$ are solution of these generalized Riccati differential equations, one first has to prove regularity, in other words the differentiability of $\Phi$ and $\psi$ with respect to time.

The theory of affine processes has been developed in several steps: in~\citet{kawazu} the full classification on the state space $ \re_+ $ was proved, introducing already the generalized Riccati equations and the related affine technology. A key step in this article is to establish the aforementioned
differentiability of the functions $\Phi$ and $\psi$ with respect to time. After several seminal papers in 
finance the classification of affine processes for the so-called canonical state space $ D=\re_+^m \times \re^{n-m} $ 
was done in~\citet{dfs}, although under the standing assumption of regularity. It remained open whether there are affine processes on the canonical state space which are not regular, or if regularity follows in fact from stochastic continuity and the property that the 
Fourier-Laplace transform is of exponential affine form. Indeed, in~\citet{kst} it is shown that affine processes on the 
canonical state space $ D=\re_+^m \times \re^{n-m} $ are regular, a reasoning motivated by insights from the solution of 
Hilbert's fifth problem, see~\citet{kst} for details. However, this solution depends on the full solution of~\cite{dfs} and 
thus on the particular polyhedral nature of the canonical state space. It remained open if regularity holds on
other ``non-polyhedral'' state spaces, for instance on sets whose boundary is described by
a parabola or on (subsets of) the cone of positive semidefinite $d \times d$ matrices, denoted by $S_d^+$. 

The following example of a possible state space illustrates that affine processes can take values in various types of sets 
and that particular geometric properties of the state space cannot be taken for granted. 
Consider the subsets of the cone of positive semidefinite $d \times d$ matrices of the form
\[
D_k=\{x \in S_d^+ \, | \, \rk(x)\leq k\}, \quad k \in \{1, \ldots, d\}.
\]
In particular, if $k \in \{1, \ldots, d-1\}$, these sets constitute \emph{non-convex} state spaces of affine processes. The non-convexity of $D_k$, $k \neq d$, is easily seen by the following argument: If $D_k$ was convex, it would contain all convex combinations of positive semidefinite matrices of rank smaller than or equal to $k$, thus also matrices of rank strictly greater than $k$, which contradicts the definition of $D_k$. Moreover, if $k \in \{1, \ldots, d-2\}$, the sets $D_k$ are maximal state spaces for affine processes in a sense made clear in the sequel. To illustrate this phenomenon by an example, let $\langle x, y\rangle:=\tr(xy)$ denote the scalar product on $S_d$, the vector space of $d \times d$ symmetric matrices, and let $d >2$ and $k \in \{1, \ldots, d-2\}$. Consider a $k \times d$ matrix of independent Brownian motions $(W_t)_{t \geq 0}$ with initial value $W_0=y \in \mathbb{R}^{k \times d}$ and define the following process
\begin{align}\label{eq:nonconvexexample}
X_t=W_t^{\top}W_t, \quad X_0=x:=y^{\top}y.
\end{align}
Then the distribution of $X_t$ corresponds to the \emph{non-central Wishart distribution} with shape parameter $\frac{k}{2}$, scale parameter $2tI$
and non-centrality parameter $x$ (see, e.g., ~\citet{letacmassam}). Its Fourier-Laplace transform is given by 
\begin{align}\label{eq:FourierLaplaceex}
\mathbb{E}_x\left[e^{\langle u, X_t\rangle}\right]=\det(I-2tu)^{-\frac{k}{2}}e^{\left\langle\frac{(I-2tu)^{-1}u+u(I-2tu)^{-1}}{2},x\right\rangle}, \quad u \in -S_d^{+}+\im S_d, 
\end{align}
and therefore of exponential affine form in all initial values $x$ with $\rk(x) \leq k$. This implies in particular that~\eqref{eq:nonconvexexample} is an affine process with state space $D_k=\{x \in S_d^+ \, | \, \rk(x)\leq k\}$ and functions $\Phi$ and $\psi$ given by
\begin{align*}
\Phi(t,u)&=\det(I-2tu)^{-\frac{k}{2}},\\
\psi(t,u)&=\frac{(I-2tu)^{-1}u+u(I-2tu)^{-1}}{2}.
\end{align*} 
Note here that the set $\mathcal{U}:=\{u \in S_d+ \im S_d\, |\, x \mapsto e^{\langle u, x\rangle } \textrm{ is bounded on } D_k\}$
corresponds to $-S_d^{+}+\im S_d$. By differentiating $\Phi$ and $\psi$ it is easily seen that these functions are solutions of the following system of Riccati ODEs
\begin{align*}
\partial_t \Phi(t,u)&=k\Phi(t,u)\langle I, \psi(t,u)\rangle, &\quad&\Phi(0,u)=1,\\
\partial_t \psi(t,u)&=2\psi(t,u)^2, &\quad& \psi(0,u)=u.
\end{align*}
From the characterization of affine processes on $S_d^+$ via the Riccati equations and the corresponding admissible parameters (see~\citet[Theorem 2.4 and Condition (2.4)]{cfmt}), we know that~\eqref{eq:FourierLaplaceex} is the Fourier-Laplace transform of an affine process with state space $S_d^+$ (meaning in particular that every starting value in $S_d^+$ is possible), if and only if $k \geq d-1$. Hence, for $k \in \{1, \ldots, d-2\}$, the state space $D_k$ cannot be enlarged to its convex hull $S_d^+$ such that the constructed affine process on $D_k$ can also be extended to an affine process on $S_d^+$. Further affine processes with state space $D_k$ can be obtained from squares of Ornstein-Uhlenbeck processes (see~\citet{bru}). 

The aim is thus to find a unified treatment which allows to prove regularity for all possible state spaces without relying on particular properties of them. In~\citet{kst1} this general question has been solved: it is shown that affine processes are regular on general state spaces $ D $, however, under the assumption that the affine process admits a c\`adl\`ag version. The method of proof is probabilistic in the sense that the ``absence of regularity'' leads -- in a probabilistic way -- to a contradiction.

This article now provides a new proof inspired by the theory of Markovian semimartingales as laid down in~\citet{cinlar}. In order to apply these reasonings, we first prove one of the last open issues in the basics of affine processes, namely that stochastic continuity and the affine property are already sufficient for the existence of a version with c\`adl\`ag trajectories, which can then be defined on the canonical probability space of c\`adl\`ag  paths with a filtration satisfying the usual conditions for any initial value. 
Let us remark, that in the existing literature on affine processes, the c\`adl\`ag property -- if addressed -- could directly be deduced from the Feller property, whose proof however strongly depends on the particular choice of the state space. Indeed, the Feller property has been shown -- under the regularity condition -- by~\citet{veerman} for state spaces of the form $\mathcal{X} \times \mathbb{R}^{n-m} $, where $\mathcal{X}\subset \mathbb{R}^m$ is a closed convex set such that the boundary of 
\[
\widetilde{\mathcal{U}}:=\{u \in \mathbb{R}^m \,|\, \sup_{x \in \mathcal{X}}\langle u, x\rangle< \infty\}
\]
is described by the zeros of a real-analytic function. In the proof, the regularity assumption is crucial to achieve this result. Otherwise the state spaces considered so far are of type $K \times \mathbb{R}^{n-m} $, where $K \subset \mathbb{R}^m$ denotes some proper convex cone. In these cases,
the Feller property and also regularity follow from the fact that the function $\psi(t,\cdot)$ maps $ -\mathring{K}^{\ast} \times \im \mathbb{R}^{n-m}$ to itself\footnote{Here, $K^{\ast}$ denotes the dual cone.} for all $t\geq 0$ and that the projection of $\psi(t,\cdot)$ on the components corresponding to the $\mathbb{R}^{n-m}$ part of the state space, denoted by $u \mapsto \Pi_{\mathbb{R}^{n-m}} \psi(t,u)$, is a linear function in $u$.
The first assertion hinges on certain order properties of the function $\Re \, \psi(t,\cdot)$ on $-K^{\ast}$, while the second one builds on the fact that $\Pi_{\mathbb{R}^{n-m}} \psi(t,\cdot)$ maps $\im \mathbb{R}^{n}$ to $\im \mathbb{R}^{n-m}$. Since we lack the mentioned order properties of $\psi$, a similar result to the first one seems hard to establish on general state spaces. 
The second one can be extended to a certain degree by considering particular sequences and projections, as done in Lemma~\ref{lem:psilin} below. In this respect the main difficulty arises from the fact that
we do not have the specific product structure of the state space at hand. For these reasons, we have to take another route, namely martingale regularization for a lot of ``test martingales", to show that affine processes admit a version with c\`adl\`ag trajectories.

Having achieved this, we proceed with the proof of regularity and provide a \emph{full} and \emph{complete} class in the sense of~\citet{cinlar} by using the process' own harmonic analysis. More precisely, we use the fact that, for all $u \in V + \im V$, for which $x \mapsto  e^{\langle u, x\rangle}$ is a bounded function on $D$, the map
\begin{align*}
x \mapsto \int_0^{\eta} \mathbb{E}_x\left[e^{\langle u,X_s\rangle}\right] \, ds, \quad \eta > 0
\end{align*}
always lies in the domain of the extended infinitesimal generator of any time-homogeneous Markov process $X$. The particular form of $\mathbb{E}_x\left[e^{\langle u,X_s\rangle}\right]$ in the case of affine processes then allows to show that the domain of the extended generator actually contains a full and complete class. This in turn implies on the one hand the semimartingale property (up to the lifetime of the affine process) and on the other hand the absolute continuity of the involved characteristics with respect to the Lebesgue measure. The final proof of regularity then builds to a large extent on these results.

The remainder of the article is organized as follows. In Section~\ref{sec:affgen} we define affine processes on general state spaces and derive some fundamental properties of the functions $\Phi$ and $\psi$. Section~\ref{sec:cadlag} and~\ref{sec:filt} are devoted to show the existence of a c\`adl\`ag  version and the right-continuity of the appropriately augmented filtration. The results on the semimartingale nature of affine process are established in Section~\ref{sec:semimart} and are used in Section~\ref{sec:regularity} for the proof of regularity.

\section{Affine Processes on General State Spaces}\label{sec:affgen}

We define affine processes as a particular class of time-homogeneous Markov processes with state space 
$D \subseteq V$, some closed, non-empty subset of an $n$-dimensional real vector space $V$ with scalar product $\langle \cdot, \cdot \rangle$. 
Symmetric matrices and the positive semidefinite matrices on $V$ are denoted by $S(V)$ and $S_+(V)$, respectively.
We write $\re_+$ for $[0,\infty)$, $\re_{++}$ for $(0,\infty)$ and $\mathbb{Q}_+$ for nonnegative rational numbers. 
For the stochastic background and notation we refer to standard text books such as~\citet{jacod} and~\citet{revuzyor}. 

To further clarify notation, we find it useful to recall in this section the basic ingredients of the theory of time-homogeneous Markov processes and the particular conventions being made in this article (compare~\citet[Chapter 1.3]{blumenthal},~\citet[Chapter 1.2]{chung},~\citet[Chapter 4]{ethier},~\citet[Chapter 3, Definition 1.1]{rogers}). 
Throughout, $D$ denotes a closed subset of $V$ and $\mathcal{D}$ its Borel $\sigma$-algebra.
Since we shall not assume the process to be conservative, we adjoin to the state space $D$ a point 
$\Delta \notin D$, called cemetery state, and set $D_{\Delta}=D\cup\{\Delta\}$ as well as 
$\mathcal{D}_{\Delta}=\sigma(\mathcal{D}, \{\Delta\})$. 
We make the convention that $\|\Delta\| :=\infty$, where $\| \cdot\|$ denotes the norm induced by the scalar product 
$\langle \cdot, \cdot \rangle$, and we set $f(\Delta)=0$ for any other function $f$ on $D$.
Moreover, in order to allow for exploding processes we shall also deal with a ``point at infinity'', denoted by $\infty$,
and $D_{\Delta} \cup \{\infty\}$ then corresponds to the one-point compactification of $D_{\Delta}$. If the state space $D$ is compact, we \emph{do not} adjoin $\{\infty\}$, since explosion is anyway not possible.

Consider the following objects on a space $\Omega$:
\begin{enumerate}
\item a stochastic process $X=(X_t)_{t\geq 0}$ taking values in $D_{\Delta}$ such that 
\begin{align}\label{eq:deltaabsorbing}
\textrm{if } X_s(\omega)=\Delta, 
\textrm{ then } X_t(\omega)=\Delta \textrm{ for all } t \geq s \textrm{ and all } \omega  \in \Omega;
\end{align}
\item the filtration generated by $X$, that is, $\mathcal{F}_t^0=\sigma(X_s,\, s\leq t)$, where we set
$\mathcal{F}^0=\bigvee_{t \in \re_+} \mathcal{F}_t^0$;
\item a family of probability measures $(\mathbb{P}_x)_{x\in D_{\Delta}}$ on $(\Omega,\mathcal{F}^0)$.
\end{enumerate}

In the course of the article, we shall show that the ``point at infinity'' $\infty$ can be identified with $\Delta$, 
since it will turn out that property~\eqref{eq:deltaabsorbing} also holds true for $\infty$ in our case.

\begin{definition}[Markov process]
A \emph{time-homogeneous Markov process}
\[
 X=\left(\Omega, (\mathcal{F}_t^0)_{t\geq 0}, (X_t)_{t\geq 0}, (p_t)_{t\geq 0}, (\mathbb{P}_x)_{x\in D_{\Delta}}\right)
\]
with state space $(D, \mathcal{D})$ (augmented by $\Delta$) is a $D_{\Delta}$-valued stochastic process such that, for all $s,t \geq 0$, $x \in D_{\Delta}$ and all bounded $\mathcal{D}_{\Delta}$-measurable functions $f: D_{\Delta} \rightarrow \re$,
\begin{align}\label{eq:Markovprop}
\mathbb{E}_x\left[f(X_{t+s})|\mathcal{F}^0_s\right]=\mathbb{E}_{X_s}\left[f(X_{t})\right]=\int_D f(\xi) p_t(X_s,d\xi), \quad \mathbb{P}_x\textrm{-a.s.}
\end{align}
Here, $\mathbb{E}_x$ denotes the expectation with respect to $\mathbb{P}_x$ and $(p_t)_{t\geq 0}$ is a \emph{transition function} on $(D_{\Delta}, \mathcal{D}_{\Delta})$. A transition function is a family of kernels 
$p_t:D_{\Delta} \times \mathcal{D}_{\Delta} \to [0,1]$ such that
\begin{enumerate}
\item for all $t\geq 0$ and $x\in D_{\Delta}$, $p_t(x, \cdot)$ is a measure on $\mathcal{D}_{\Delta}$ with $p_t(x, D) \leq 1$, 
$p_t(x,\{\Delta\})=1-p_t(x,D)$ and $p_t(\Delta,\{\Delta\})=1$;
\item\label{item:normal} for all $x \in D_{\Delta}$, $p_0(x,\cdot)=\delta_{x}(\cdot)$, where $\delta_{x}(\cdot)$ denotes the Dirac measure at $x$;
\item\label{item:measurable} for all $t\geq 0$ and $\Gamma \in \mathcal{D}_{\Delta}$, $x \mapsto p_t(x, \Gamma)$ is $\mathcal{D}_{\Delta}$-measurable;
\item for all $s,t\geq 0$,  $x\in D_{\Delta}$ and $\Gamma \in \mathcal{D}_{\Delta}$, the Chapman-Kolmogorov equation holds, that is,
\[
 p_{t+s}(x,\Gamma)=\int_{D_{\Delta}} p_s(x,d\xi)p_t(\xi,\Gamma).
\]
\end{enumerate}
If $(\mathcal{F}_t)_{t\geq 0}$ is a filtration with $\mathcal{F}_t^0 \subset \mathcal{F}_t$, $t \geq 0$, then $X$ is a 
\emph{time-homogeneous Markov process relative to $(\mathcal{F}_t)$} if~\eqref{eq:Markovprop} holds with $\mathcal{F}_s^0$ replaced by
$\mathcal{F}_s$.
\end{definition}

We can alternatively think of the transition function as inducing a measurable contraction semigroup  
$(P_t)_{t \geq 0}$ defined by
\[
 P_tf(x):=\mathbb{E}_{x}[f(X_{t})]=\int_{D} f(\xi)p_t(x,d\xi), \quad x \in D_{\Delta},
\]
for all bounded $\mathcal{D}_{\Delta}$-measurable functions $f: D_{\Delta} \rightarrow \re$.

\begin{remark}
\begin{enumerate}
\item Note that, in contrast to~\citet{dfs}, we do not assume $\Omega$ to be the canonical space of all functions $\omega: \re_+ \to D_{\Delta}$,
but work on some general probability space. 
\item Since we have $p_t(x, \Gamma)=\mathbb{P}_x[X_t \in \Gamma]$  for  all $t \geq 0$, $x\in D_{\Delta}$ and 
$\Gamma \in \mathcal{D}_{\Delta}$, property~(ii) 
and~(iii) of the transition function, imply $\mathbb{P}_x[X_0=x]=1$ for all $x \in D_{\Delta}$ and 
measurability of the map $x \mapsto \mathbb{P}_x[X_t \in \Gamma]$ for  all $t \geq 0$  and $\Gamma \in \mathcal{D}_{\Delta}$. 
\end{enumerate}
\end{remark} 

For the following definition of affine processes, let us introduce the set $\mathcal{U}$ defined by
\begin{align}\label{eq:U}
\mathcal{U}=\left\{u \in V +\im V\,\big|\, e^{\langle u,x\rangle} \textrm{ is a bounded function on $D$}\right\}.
\end{align}
Clearly $ \im V \subseteq \mathcal{U}$. Here, the set $\im V$ stands for purely imaginary elements 
and $\langle \cdot, \cdot \rangle$ is the extension of the real scalar product to $V +\im V$, but without complex conjugation.
Moreover, we denote by $p$ the dimension of $\Re \,\mathcal{U}$ and write $\langle \Re\, \mathcal{U} \rangle$ for its (real) linear hull and  
$ {\langle \Re\, \mathcal{U} \rangle}^\perp $ for its orthogonal complement in $V $. 
The set $\im  {\langle \Re\, \mathcal{U} \rangle}^\perp \subset \mathcal{U} $ corresponds to the purely imaginary directions of $ \mathcal{U} $. 
Finally, for some linear subspace $W \subset V $, $\Pi_{W}: V \to V $ denotes the orthogonal projection on $W$, 
which is extended to $ V + \im V $ by linearity, i.e.,~$\Pi_W(v_1 + \im v_2):=\Pi_Wv_1 + \im \Pi_W v_2 $.

Furthermore we need the sets
\[
\mathcal{U}^m=\left\{u \in V + \im V\, |\, \sup_{x \in D}e^{\langle \Re u, x\rangle} \leq m \right\}, \quad m \geq 1,
\]
and note that $\mathcal{U}=\bigcup_{m \geq 1} \mathcal{U}^m$ and $\im V \subseteq \mathcal{U}^m$ for all $m\geq 1$.

\begin{assumption}\label{ass:statespace}
Recall that $\dim V=n$. We suppose that the state space $D$ contains at least 
$n+1$ affinely independent elements $x_1,\ldots,x_{n+1}$, that is, 
the $n$ vectors $(x_1-x_j, \ldots, x_{j-1}-x_j, x_{j+1}-x_j, \ldots, x_{n+1}-x_j)$ 
are linearly independent for every $j \in \{1, \ldots, n+1\}$.
\end{assumption}

We are now prepared to give our main definition:

\begin{definition}[Affine process]\label{def:affineprocess}
A time-homogeneous Markov process $X$ relative to some filtration $(\mathcal{F}_t)$ and with state space $(D,\mathcal{D})$ (augmented by $\Delta$) is called \emph{affine} if
\begin{enumerate}
\item it is stochastically continuous, that is, $\lim_{s\to t}
p_s(x,\cdot)=p_t(x,\cdot)$ weakly on $D$ for every $t \geq 0$ and $x\in D$, and
\item\label{def:affineprocess2} its Fourier-Laplace transform has exponential affine
dependence on the initial state. This means that there exist functions 
$\Phi:\re_+ \times \mathcal{U} \to \cn$ and $\psi:\re_+ \times \mathcal{U} \to V +\im V$ such that, for every $x \in D$ and $m \geq 1$, the map $(t,u) \mapsto \langle \psi(t,u),x\rangle $ is locally continuous on the subset of $ \re_+ \times \mathcal{U}^m $ where $ \Phi $ does not vanish, and 
\begin{align}\label{eq:affineprocess}
\mathbb{E}_x\left[e^{\langle u, X_t\rangle} \right]=P_te^{\langle u, x\rangle}=\int_{D}e^{\langle u, \xi
\rangle}p_t(x,d\xi)=\Phi(t,u)e^{\langle \psi(t,u),x\rangle},
\end{align}
for all $x\in D$ and $(t,u) \in \re_+ \times \mathcal{U}$.
\end{enumerate}
\end{definition}

\begin{remark}\label{rem:definitionaffine}
\begin{enumerate}
\item The above definition differs in four crucial details from the definitions given
in~\citet[Definition 2.1, Definition 12.1]{dfs}.\footnote{In Definition 2.1 affine processes on the canonical state space $D=\re_+^m \times \re^{n-m}$ are considered,
whereas in Definition 12.1 the state space $D$ can be an arbitrary subset of $\re^n$.}
\begin{enumerate}
\item 
First, therein the right hand side of~\eqref{eq:affineprocess} 
is defined in terms of a function $\phi(t,u)$, namely as $e^{\phi(t,u)+\langle \psi(t,u),x\rangle}$, 
such that the function $\Phi(t,u)$ in our definition corresponds to $e^{\phi(t,u)}$.
Our definition is in line with the one given in~\citet{kawazu} and~\citet{kst, kst1} and 
differs from the one in~\citet{dfs}, as we do not require $\Phi(t,u) \neq 0$ a priori.
However, since all affine processes on $D=\re_+^m \times \re^{n-m}$ are infinitely divisible (see~\citet[Theorem 2.15]{dfs}), it turns out 
with hindsight that setting $\Phi(t,u)= e^{\phi(t,u)}$ is actually no restriction.
\item 
Second, we assume that the affine property~\eqref{eq:affineprocess} holds for all $u \in \mathcal{U}$, whereas 
on the canonical state space $D=\re_+^m \times \re^{n-m}$ it is restricted to $\im \re^n$ (see~\citet{dfs}).
This however turns out to imply the affine property~\eqref{eq:affineprocess} also on $\mathcal{U}$.
\item
Third, in contrast to~\citet{dfs}, we take stochastic continuity as part of the
definition of an affine process. We remark that there are simple examples of Markov processes which
satisfy Definition~\ref{def:affineprocess}~(ii),  
but are not stochastically continuous (see~\citet[Remark~2.11]{dfs}). 
\item 
Fourth, due to the general structure of the state space $D$, we decided to assume local continuity of $(t,u) \mapsto \langle \psi(t,u),x \rangle$ on the subset of $ \re_+ \times \mathcal{U}^m $ where $ \Phi $ does not vanish, which we denote by 
$\mathcal{Q}^m=\{(t,u) \in \re_+ \times \mathcal{U}^m \, | \, \Phi(s,u)\neq 0, \textrm{ for all } s \in [0,t]\}$
in the sequel. This condition could be replaced by the following weaker requirement: 
For every $m \geq 1$ and all $(t_0, u_0,x)\in \mathcal{Q}^m \times D$, there exist some neighborhood $U$ such that for all $(t,u) \in U$
\begin{align}\label{eq:psiclose}
| \langle \Im\, \psi(t,u),x \rangle-\langle \Im\, \psi(t_0,u_0),x \rangle| < \pi.
\end{align}
Indeed, in order to conclude the existence of a unique continuous choice for $\Phi$ and $\psi$ on $\mathcal{Q}^m$, this is the only condition needed in the proof of Proposition~\ref{prop:PhiPsiproperties} below. Notice that in many cases the mere existence of $ \Phi $ and $ \psi $ satisfying \eqref{eq:affineprocess} is sufficient for the existence of a continuous selection, e.g., for star shaped spaces $ D $.
\end{enumerate}

\item Let us remark that the assumption of a \emph{closed} state space is no restriction. Indeed, if an affine process is defined 
on some state space $D$, which is only supposed to be an arbitrary Borel subset of $V$ as done in~\citet{kst1}, 
then the affine property~\eqref{eq:affineprocess} extends automatically to $\overline{D}$: 
Let $(x_k)_{k \in \mathbb{N}}$ be a sequence in $D$ converging to some $x \in  \overline{D}$. 
Due to the exponential affine form of the characteristic function, we have for all $t \in \re_+$ and $u \in \im V$
\[
\mathbb{E}_{x_n}\left[e^{\langle u, X_t \rangle}\right]= \Phi(t,u)e^{\langle \psi(t,u), x_n \rangle} \to  \Phi(t,u)e^{\langle \psi(t,u), x \rangle}.
\]
Since the left hand side is continuous in $u$, the same holds true for the right hand side. 
Whence L\'evy's continuity theorem implies that the right hand side is a characteristic function of some substochastic measure 
$p_t(x, \cdot)$ on $\overline{D}$, which is the weak limit of $p_t(x_n, \cdot)$. 
As stochastic continuity and the Chapman-Kolmogorov equations extend to $\overline{D}$, and since weak convergence implies the convergence of the Fourier-Laplace transforms on $\mathcal{U}$, we thus have constructed an affine process with state space $\overline{D}$.

\item Note furthermore that Assumption~\ref{ass:statespace} is no restriction, since we can always 
pass to a lower dimensional ambient vector space if $D$ does not contain $n+1$ affinely independent elements. 
Moreover, note also that we do not exclude compact state spaces. For examples of affine processes on compact state spaces we refer to Remark~\ref{rem:compactstatespace}.

\item We finally remark that in Section~\ref{sec:cadlag} we consider affine processes on the filtered space $(\Omega, \mathcal{F}^0, \mathcal{F}_t^0)$, where $\mathcal{F}^0_t$ denotes the natural filtration and 
$\mathcal{F}^0=\bigvee_{t \in \re_+} \mathcal{F}_t^0$, as introduced above. 
However, we shall progressively enlarge the filtration by augmenting with the respective
null-sets.
\end{enumerate}
\end{remark}

\begin{proposition}\label{prop:PhiPsiproperties}
Let $X$ be an affine process relative to some filtration $(\mathcal{F}_t)$. 
Then we have the following properties:
\begin{enumerate}
\item\label{item:jointcont} If we set $\Phi(0,u)=1$ and  $\psi(0,u)=u$ for all $u \in \mathcal{U}$, then there is a unique choice $\Phi$ and $\psi$ in \eqref{eq:affineprocess} such that $ \Phi $, $\psi $ are jointly continuous on $\mathcal{Q}^m=\{(t,u) \in \re_+ \times \mathcal{U}^m \, | \, \Phi(s,u)\neq 0, \textrm{ for all } s \in [0,t]\}$ for $ m \geq 1 $. 
\item\label{item:rangePsi} $\psi$ maps the set $\mathcal{O}=\{(t,u) \in \re_+ \times \mathcal{U}\, | \, \Phi(t,u)\neq 0\}$ to $\mathcal{U}$. 
\item\label{item:semiflowPhiPsi} The functions $\Phi$ and $\psi$ satisfy the \emph{semiflow} property: Let $u \in \mathcal{U}$ and $t,s \geq 0$. 
Suppose that $\Phi(t+s,u)\neq 0$, then also
$\Phi(t,u) \neq 0$ and $\Phi(s, \psi(t,u)) \neq 0$ and we have
\begin{equation}
\begin{split}\label{eq:flowprop}
\Phi(t+s,u) &= \Phi(t,u)\Phi(s,\psi(t,u)),\\
\psi(t+s,u) &= \psi(s,\psi(t,u)).
\end{split}
\end{equation}
\end{enumerate}
\end{proposition}

\begin{proof}

Fix $ m \geq 1 $. It follows e.g.~from~\citet[Lemma 23.7]{bau_96} that stochastic continuity of $X$ implies joint continuity of 
$(t,u) \mapsto P_te^{\langle u,x \rangle}$ on $\mathbb{R}_+ \times \mathcal{U}^m$ for all $x \in D$. 
Hence $(t,u)\mapsto \Phi(t,u)e^{\langle \psi(t,u),x\rangle}$ is jointly continuous on $\mathbb{R}_+ \times \mathcal{U}^m$ for every $x \in D$. 
By Assumption~\ref{ass:statespace} on the state space $D$, this in turn
yields a unique continuous choice of the functions $(t,u) \mapsto \Phi(t,u)$ and $(t,u) \mapsto\psi(t,u)$ on 
$\mathcal{Q}^m$. Indeed, by~\citet[Proposition 2.4]{kst1}, we know that for every $x \in D$ there exists a unique continuous 
logarithm $g(x;\cdot,\cdot): \mathcal{Q}^m \to \mathbb{C}$, $(t,u)\mapsto g(x; t,u)$ such that
for all $(t,u,x) \in \mathcal{Q}^m \times D$
\[
e^{g(x;t,u)}=\Phi(t,u)e^{\langle \psi(t,u), x\rangle},
\]
with $g(x;0,0)=0$ holds true. Without loss of generality we suppose $0 \in D$, then it follows that
\[
\Phi(t,u)=e^{g(0;t,u)}=:e^{\phi(t,u)},
\]
is continuous in $(t,u)$. Setting $h(x;t,u):=g(x;t,u)-\phi(t,u)$, with $ \phi(t,u,):=g(0;t,u) $, we have for all $(t,u,x) \in \mathcal{Q}^m \times D$
\begin{align}\label{eq:log}
e^{h(x;t,u)}=e^{\langle \psi(t,u), x\rangle},
\end{align}
whence 
\begin{align}\label{eq:h}
h(x;t,u)=\langle \psi(t,u), x\rangle+ 2\pi \im k(t,u,x), \quad  k(t,u,x)\in \mathbb{Z}.
\end{align}
Moreover, the local continuity assumption on $(t,u)\mapsto \langle \psi(t,u),x \rangle $ implies that for all $(t_0,u_0,x)$ there exists some neighborhood around $(t_0,u_0)$ such that
\[
(t,u) \mapsto \langle \psi(t,u), x\rangle
\]
is continuous.\footnote{Due to relation~\eqref{eq:log} and the continuity of $(t,u)\mapsto h(x;t,u)$, this is also implied by the weaker condition~\eqref{eq:psiclose}.} 
Since $k(t,u,x)\in \mathbb{Z}$, it follows that $k(t,u,x)=k(x)$ on $\mathcal{Q}^m$. 
Setting $t=0$ and $u=0$ in~\eqref{eq:h} thus yields for all $x \in D$
\[
h(x;0,0)-2\pi \im k(x)= -2\pi \im k(x)=0, 
\]
and in particular a unique continuous specification of $(t,u) \mapsto \psi(t,u)$, since 
\[ 
{h(x;t,u)}={\langle \psi(t,u), x\rangle} 
\]
can be uniquely solved for $ \psi $ due to Assumption~\ref{ass:statespace}. The choice of $ \Phi $ and $ \psi $ certainly does not depend on $ m $ but only on the initial conditions for $ t = 0 $.

Concerning~\ref{item:rangePsi}, let $(t,u) \in \mathcal{O}=\{(t,u) \in \re_+ \times \mathcal{U} \, | \, \Phi(t,u)\neq 0\}$. Since
\[
\left|\Phi(t,u)e^{\langle\psi(t,u), x\rangle}\right|=\left|\mathbb{E}_x\left[e^{\langle u, X_t\rangle}\right]\right|\leq \mathbb{E}_x\left[\left|e^{\langle u, X_t\rangle}\right|\right]
\]
is bounded on $D$ and as $\Phi(t,u)\neq 0$, we conclude that $\psi(t,u) \in \mathcal{U}$.

Assumption $\Phi(t+s,u)\neq 0$ in~\ref{item:semiflowPhiPsi} implies
\begin{align}\label{eq:expnon0}
\mathbb{E}_x\left[e^{\langle u, X_{t+s}\rangle}\right] =\Phi(t+s,u)e^{\langle\psi(t+s,u), x\rangle}\neq 0.
\end{align}
By the law of iterated expectations and the Markov property, we thus have
\begin{align}\label{eq:semiflowproof}
\mathbb{E}_x\left[e^{\langle u, X_{t+s}\rangle}\right] =
\mathbb{E}_x\left[\mathbb{E}_x\left[e^{\langle u, X_{t+s}\rangle}\Big | \mathcal{F}_s\right ]\right]
=\mathbb{E}_x\left[\mathbb{E}_{X_s}\left[e^{\langle u, X_{t}\rangle} \right]\right].
\end{align}
If $\Phi(t,u)=0$ or $\Phi(s, \psi(t,u))=0$, then the inner or the outer expectation evaluates to $0$. This implies that
the whole expression is $0$, which contradicts~\eqref{eq:expnon0}. Hence $\Phi(t,u)\neq0$ and 
$\Phi(s, \psi(t,u))\neq 0$ and we can write~\eqref{eq:semiflowproof} as
\[
 \mathbb{E}_x\left[e^{\langle u, X_{t+s}\rangle}\right ]=\mathbb{E}_x\left[\Phi(t,u)e^{\langle\psi(t,u), X_s\rangle}\right ] =
\Phi(t,u)\Phi(s, \psi(t,u))e^{\langle\psi(s,\psi(t,u)), x\rangle}.
\]
Comparing with~\eqref{eq:expnon0} yields the claim by uniqueness of $ \Phi $ and $ \psi $.
\end{proof}

\begin{remark}\label{rem:jointcont}
Henceforth, the symbols $\Phi$ and $\psi$ always correspond to the unique choice as established in 
Proposition~\ref{prop:PhiPsiproperties}.
\end{remark}

\section{Affine processes have a c\`adl\`ag Version}\label{sec:cadlag}

The aim of this section is to show that the definition of an affine process already implies the existence of
a c\`adl\`ag version. This is the core section of this article and of a remarkable subtlety, 
which is maybe less surprising if one considers the generality of the question.
So far we do not know whether general affine processes are Feller processes. If the Feller property held true, this would allow us to conclude the existence of a c\`adl\`ag version. Moreover, we also cannot apply the most general standard criteria for the existence of c\`adl\`ag versions, as for instance described 
in \cite[Theorem I.6.2]{gihsko:83}.

Our approach to the problem is inspired by martingale regularization for a lot of ``test martingales'', from which we want to conclude path properties of the original stochastic process. The main difficulty here is that explosions and/or killing might appear.

Indeed, for every fixed $x \in D$, we first establish that for $\mathbb{P}_x$-almost every $\omega$
\[
t \mapsto M_t^{T,u}(\omega):=\Phi(T-t,u)e^{\langle\psi(T-t,u),X_t(\omega)\rangle}, \quad t \in [0,T],
\]
is the restriction to $ \mathbb{Q}_+ \cap [0,T]$ of a c\`adl\`ag function for almost all $(T,u) \in (0,\infty) \times  \mathcal{U}$, in the sense that
$M_t^{T,u}=0$ if $\Phi(T-t,u)=0$.
This is an
application of Doob's regularity theorem for supermartingales, where we can conclude -- using Fubini's theorem -- that there exists a $\mathbb{P}_x$-null-set outside of which we observe appropriately regular trajectories for almost all $(T,u)$. 

\begin{proposition}\label{prop:Mcadlag}
Let $x \in D$ be fixed and let $X$ be an affine process relative to $(\mathcal{F}_t^0)$. Then 
\[
\lim_{\substack{q \in \mathbb{Q}_+\\ q \downarrow t}} M_q^{T,u}=\lim_{\substack{q \in \mathbb{Q}_+\\ q \downarrow t}}\Phi(T-q,u)e^{\langle\psi(T-q,u),X_q\rangle}, \quad t \in [0,T],
\]
exists $\mathbb{P}_x$-a.s.~for almost all $(T,u) \in (0,\infty) \times \mathcal{U}$ and defines a c\`adl\`ag function in $t$.
\end{proposition}

\begin{proof}
In order to prove this result, we adapt parts of the proof of~\citet[Theorem I.4.30]{protter} to our setting.
Due to the law of iterated expectations
\[
M_t^{T,u}=\Phi(T-t,u)e^{\langle\psi(T-t,u),X_t\rangle}=\mathbb{E}_x\left[e^{\langle u, X_T\rangle} \big | \mathcal{F}^0_t\right], \quad t \in [0,T],
\]
is a (complex-valued) $(\mathcal{F}_t^0, \mathbb{P}_x)$-martingale for every $u \in \mathcal{U}$ and every $T >0$. 
From Doob's regularity theorem (see, e.g.,~\citet[Theorem II.65.1]{rogers}) it then follows that, for any fixed $(T,u)$,
the function $t \mapsto M_t^{T,u}(\omega)$, with $t \in \mathbb{Q}_+ \cap [0,T]$, is the restriction to $ \mathbb{Q}_+ \cap [0,T]$ 
of a c\`adl\`ag function for $\mathbb{P}_x$-almost every $\omega$.
Define now the set
\begin{multline}\label{eq:Gamma}
 \Gamma=\{(\omega, T,u) \in \Omega \times (0,\infty) \times \mathcal{U}\, | \, t \mapsto M_t^{T,u}(\omega),\, t\in \mathbb{Q}_+ \cap [0,T],\\
\textrm{ is not the restriction of a  c\`adl\`ag function}\}.
\end{multline}
Then $\Gamma$ is a $\mathcal{F}^0 \otimes \mathcal{B}((0,\infty) \times \mathcal{U})$-measurable set. Due to the above argument
concerning regular versions of (super-)martingales, $\int_{\Omega} 1_{\Gamma}(\omega,T,u) \mathbb{P}_x(d\omega)=0$ for any 
$(T,u) \in (0,\infty) \times \mathcal{U}$. By Fubini's theorem, we therefore have
\[
 \int_{\Omega}\int_{(0,\infty) \times \mathcal{U}}1_{\Gamma}(\omega,T,u)d\lambda\,\mathbb{P}_x(d\omega)= 
\int_{(0,\infty) \times \mathcal{U}}\int_{\Omega}1_{\Gamma}(\omega,T,u)\mathbb{P}_x(d\omega)\,d\lambda=0,
\]
where $\lambda$ denotes the Lebesgue measure. 
Hence, for $\mathbb{P}_x$-almost every $\omega$, $t \mapsto M_t^{T,u}(\omega)$, the map $t\in \mathbb{Q}_+ \cap [0,T]$, is the restriction of a c\`adl\`ag function for $\lambda$-almost all $(T,u) \in (0,\infty) \times \mathcal{U}$, which proves the result.
\end{proof}

Having established path regularity of the martingales $M^{T,u}$, we want to deduce the same for the affine process $X$. This is the purpose of the subsequent lemmas and propositions, 
for which we need to introduce the following sets: 
\begin{align}
\widetilde{\Omega} &\textrm{ is the projection of } \{\Omega \times (0,\infty) \times \im V\} \setminus \Gamma  \textrm{ onto } \Omega,\label{eq:omegatilde}\\
\mathcal{T} &\textrm{ is the projection of } \{\Omega \times (0,\infty) \times \im V\} \setminus \Gamma  \textrm{ onto } (0, \infty),\label{eq:mathcalT}\\
\mathcal{V} &\textrm{ is the projection of } \{\Omega \times (0,\infty) \times \mathcal{U}\} \setminus \Gamma  \textrm{ onto } \mathcal{U},\label{eq:mathcalV}\\
\mathcal{V}^m &\textrm{ is the projection of } \{\Omega \times (0,\infty) \times \mathcal{U}^m\} \setminus \Gamma  \textrm{ onto } \mathcal{U}^m,\label{eq:mathcalVm}
\end{align}
where $\Gamma$ is given in~\eqref{eq:Gamma}. Denoting by $\mathcal{F}^x$ the completion of $\mathcal{F}^0$ with respect to 
$\mathbb{P}_x$, let us remark that the measurable projection theorem implies that $\widetilde{\Omega} \in \mathcal{F}^x$ and by the above proposition we have $\mathbb{P}_x[\widetilde{\Omega}]=1$. 

The following lemma is needed to prove Proposition~\ref{prop:Xcadlag} below which is essential for establishing the existence of 
a c\`adl\`ag version of $X$. 

\begin{lemma} \label{lem:psilin} 
 
Let $\psi$ be given by~\eqref{eq:affineprocess} and assume that there exists some $D$-valued sequence $(x_k)_{k \in \mathbb{N}}$ such that 
\begin{align}\label{eq:realcomp}
\lim_{k \to \infty} \Pi_{\langle \Re \, \mathcal{U} \rangle}x_k=:\lim_{k \to \infty} y_k
\end{align}
exists finitely valued and 
\begin{align}\label{eq:imcomp}
\limsup_{k \to \infty} \|\Pi_{  {\langle \Re\, \mathcal{U} \rangle}^\perp} x_k\| =\infty.
\end{align}
\begin{enumerate}
\item
Then we can choose a subsequence of $ (x_k) $ denoted again by $ (x_k) $: 
along this sequence there exist a finite number of mutually 
orthogonal directions $ g_i \in {\langle \Re\, \mathcal{U} \rangle}^\perp $ 
of length $1$ such that
\[
x_k - \sum_i \langle x_k, g_i \rangle g_i 
\]
converges as $ k \to \infty $ and $ \langle x_k, g_i \rangle $ diverges as $ k \to \infty $, where the rates of divergence are non-increasing in $i$ in the sense that 
\[
\limsup_{k \to \infty} \frac{\langle x_k, g_{i+1} \rangle}{\langle x_k, g_i \rangle} < \infty.
\] 
\item Moreover, let $T >0$ be fixed and let $r > 0$ such that $\Phi(t,u) \neq 0$ for all $(t,u) \in [0,T]\times B_r$, where $B_r$ denotes the ball with radius $r$ in $\im V$, i.e.
\[
B_r=\{u \in \im V\, |\, \|u\| < r\}.
\] 
Then, there exist continuous functions $R: [0,T] \to \re_{++}$ and $\lambda_i: [0,T] \to V$ such that
\[
\langle \psi(t,u) , g_i \rangle =\langle \lambda_i(t),u\rangle
\] 
for all $u \in B_{R(t)}$.
\end{enumerate}
\end{lemma}

\begin{proof}
Concerning the first assertion, we define -- by choosing appropriate subsequences, still denoted  by $(x_k)$ -- 
the directions of divergence in $\langle \Re \, \mathcal{U} \rangle ^{\perp}$ inductively by
\begin{align}\label{eq:gr}
g_r=\lim_{k \to \infty} \frac{x_k-\sum_{i=1}^{r-1} \langle x_k,g_i\rangle g_i}{\|x_k-\sum_{i=1}^{r-1}\langle x_k,g_i\rangle g_i\|}
\end{align}
as long as $\limsup_{k \to \infty}\|x_k-\sum_{i=1}^{r-1} \langle x_k,g_i\rangle g_i\| = \infty$. Notice that we can
choose the directions $ g_i $ mutually orthogonal and the rates of divergence $ \langle g_i , x_k \rangle $ non-increasing
in $ i $.

For the second part of the statement, we adapt the proof of~\citet[Lemma 3.1]{kst} to our situation, 
using in particular the existence of a sequence in $D$ with the properties~\eqref{eq:realcomp} and~\eqref{eq:imcomp}.
As characteristic function, the map
$ \im V \ni u \mapsto \mathbb{E}_x[e^{\langle u, X_{t} \rangle}]$ is positive definite for any $x \in D$ and 
$t \geq 0$. Define now for every 
$u \in  B_r$, $x \in D$ and $t \in [0,T]$ the function
\begin{align}\label{eq:Thetachar}
\Theta(u,t,x)=\frac{\mathbb{E}_x\left[e^{\langle u, X_{t} \rangle}\right]}{\Phi(t,0)e^{\langle  \Pi_{\langle \Re \, \mathcal{U}\rangle}\psi(t,0), \Pi_{\langle \Re \, \mathcal{U}\rangle}x \rangle}}=\frac{\Phi(t,u)e^{\langle \psi(t,u), x \rangle}}{\Phi(t,0)e^{\langle  \Pi_{\langle \Re \, \mathcal{U}\rangle}\psi(t,0), \Pi_{\langle \Re \, \mathcal{U}\rangle}x \rangle}}.
\end{align}
As $\mathbb{E}_x\left[e^{\langle 0, X_t\rangle}\right]=\Phi(t,0) e^{\langle \psi(t,0),x \rangle }$ is real-valued and positive 
for all $t \in [0,T]$, we conclude -- due to Assumption~\ref{ass:statespace} and the continuity of the functions $t \mapsto \Phi(t,0)$ and $t \mapsto \psi(t,0)$ -- that $\Im \Phi(t,0)=0$ and $\Im \, \psi(t,0)=0$ for all $t \in [0,T]$.
In particular, the denominator in~\eqref{eq:Thetachar} is positive, which implies that $B_r \ni u \mapsto \Theta
(u,t,x)$ is a positive definite function for all $t \in [0,T]$ and $x \in D$. Moreover, 
since $\Pi_{ {\langle \Re\, \mathcal{U} \rangle}^\perp} \psi(t,0)$ is purely imaginary and thus in particular $0$ for all $t\in [0,T]$, it follows that 
\[
\Theta(0,t,x)=\exp\left(\langle \Pi_{ {\langle \Re\, \mathcal{U} \rangle}^\perp} \psi(t, 0), \Pi_{ {\langle \Re\, \mathcal{U} \rangle}^\perp} x\rangle\right)=1
\]
for all $t \in [0,T]$ and $x \in D$. This together with the positive definiteness of $u \mapsto \Theta(u,t,x)$ yields 
\begin{align}\label{eq:Thetaposdef}
| \Theta(u+v,t,x)-\Theta(u,t,x)\Theta(v,t,x)|^2 \leq 1, \quad u,v   \in B_{\frac{r}{2}},\, t \geq 0, x \in D.
\end{align}
Indeed, this inequality is obtained by computing the determinant of the positive semidefinite matrix
\[
\begin{pmatrix}
\Theta(0,t,x) & \overline{\Theta(u,t,x)} &\Theta(v,t,x)\\
\Theta(u,t,x) &  \Theta(0,t,x)  & \Theta(u+v,t,x)\\
\overline{\Theta(v,t,x)}& \overline{\Theta(u+v,t,x)} & \Theta(0,t,x)
\end{pmatrix}
\]
(compare, e.g.,~\citet[Lemma 3.2]{kst}).
Let us now define $y:=\Pi_{\langle \Re \, \mathcal{U}\rangle}x$ and
\begin{align*}
Z_1(u,v,y,t)&=\frac{\Phi(t,u+v)e^{\langle  \Pi_{\langle \Re \, \mathcal{U}\rangle}\psi(t,u+v), y\rangle}}{\Phi(t,0)e^{\langle  \Pi_{\langle \Re \, \mathcal{U}\rangle}\psi(t,0), y \rangle}}, \\
Z_2(u,v,y,t)&=\frac{\Phi(t,u)\Phi(t,v)e^{\langle  \Pi_{\langle \Re \, \mathcal{U}\rangle}(\psi(t,u)+\psi(t,v)), y \rangle}}{\Phi(t,0)^2e^{2\langle  \Pi_{\langle \Re \, \mathcal{U}\rangle}\psi(t,0),y \rangle}},\\
\beta_1(u,v,t)&=\Im (\Pi_{{\langle \Re\, \mathcal{U} \rangle}^\perp}\psi(t,u+v)),\\ 
\beta_2(u,v,t)&=\Im (\Pi_{{\langle \Re\, \mathcal{U} \rangle}^\perp}\psi(t,u))+\Im (\Pi_{{\langle \Re\, \mathcal{U} \rangle}^\perp}\psi(t,v)),\\
r_1(u,v,y,t)&=|Z_1|=\left| \frac{\Phi(t,u+v)}{\Phi(t,0)}\right|e^{\langle \Re (\Pi_{\langle \Re \, \mathcal{U} \rangle } (\psi(t, u+v)-\psi(t,0))),y \rangle},\\ 
r_2(u,v,y,t)&=|Z_2|=\left| \frac{\Phi(t,u)\Phi(t,v)}{\Phi(t,0)^2}\right|e^{\langle \Re (\Pi_{\langle \Re \, \mathcal{U} \rangle } (\psi(t, u)+\psi(t,v)-2\psi(t,0))),y \rangle},\\
\alpha_1(u,v,y,t)&=\arg(Z_1)= \arg \left(\frac{\Phi(t,u+v)}{\Phi(t,0)}\right)\\
&\quad \quad \quad \quad \quad \quad+\langle \Im( \Pi_{\langle \Re \, \mathcal{U} \rangle } \psi(t, u+v)),y \rangle, \\
\alpha_2(u,v,y,t)&=\arg(Z_2)=\arg \left(\frac{\Phi(t,u)\Phi(t,v)}{\Phi(t,0)^2}\right)\\
&\quad \quad \quad \quad \quad \quad+\langle \Im (\Pi_{\langle \Re \, \mathcal{U} \rangle } (\psi(t, u)+\psi(t,v)),y \rangle.
\end{align*}
Using~\eqref{eq:Thetaposdef} and the fact that $2r_1r_2 \leq r_1^2+r_2^2$, we then obtain
\begin{align*}
1&\geq\left|r_1e^{\im(\alpha_1+\langle \beta_1,\Pi_{{\langle \Re \, \mathcal{U} \rangle}^{\perp}} x\rangle)}-r_2e^{\im (\alpha_2+\langle \beta_2,\Pi_{{\langle \Re \, \mathcal{U} \rangle}^{\perp}} x\rangle)}\right|^2\\
&=r_1^2+r_2^2-2r_1r_2\cos(\alpha_1-\alpha_2+\langle \beta_1-\beta_2, \Pi_{ {\langle \Re\, \mathcal{U} \rangle}^\perp} x \rangle)\\
&\geq 2r_1r_2(1-\cos(\alpha_1-\alpha_2+\langle \beta_1-\beta_2, \Pi_{ {\langle \Re\, \mathcal{U} \rangle}^\perp} x \rangle)),
\end{align*}
whence
\begin{multline}\label{eq:contra}
r_1(u,v,y,t)r_2(u,v,y,t)\\
\times (1-\cos(\alpha_1(u,v,y,t)-\alpha_2(u,v,y,t)+\langle \beta_1(u,v,t)-\beta_2(u,v,t), \Pi_{ {\langle \Re\, \mathcal{U} \rangle}^\perp} x \rangle)) \leq \frac{1}{2}.
\end{multline}
Define now 
\begin{align*}
R(t,y)=\sup\Big\{\rho \in \left[0,\frac{r}{2}\right] \, | \, &r_1(u,v,y,t)r_2(u,v,y,t) > \frac{3}{4} \textrm{ for $u,v \in B_{\frac{r}{2}}$} \\
&\textrm{ with $\| u\|\leq \rho$ and $\|v \| \leq \rho$}\Big\}.
\end{align*}
Note that $R(t,y)> 0$ for all $(t,y) \in [0,T]\times \Pi_{\langle \Re \, \mathcal{U}\rangle}D$, which follows from 
the fact that $r_1(0,0,y,t)=r_2(0,0,y,t)=1$ and the continuity of 
\[
(u,v) \mapsto r_1(u,v,y,t)r_2(u,v,y,t).
\]
Continuity of $(t,y) \mapsto r_1(u,v,y,t)r_2(u,v,y,t)$ also implies that $(t,y) \mapsto R(t,y)$ is continuous. Set now $R(t):=\inf_{k} R(t,y_k)$ where $y_k=\Pi_{\langle \Re \, \mathcal{U}\rangle}x_k$. Then~\eqref{eq:realcomp} implies that $R(t)> 0$ for all $t \in [0,T]$. 

Let now $t$ be fixed and $g_1$ given by~\eqref{eq:gr}. Suppose that
\[
\langle \beta_1(u^{\ast},v^{\ast},t)-\beta_2(u^{\ast},v^{\ast},t), g_1 \rangle \neq 0
\]
for some $u^{\ast},v^{\ast} \in B_{R(t)}$. Then due to the continuity of $\beta_1$ and $\beta_2$, there exists some 
$\delta >0$ such that for all $ u, v $ in a neighborhood  $O_{\delta}$ of $(u^{\ast}, v^{\ast})$ defined by
\begin{align*}
 O_{\delta}=\Big\{u,v \in B_{R(t)} \,|\, &\|u-u^{\ast}\|< \delta, \|v-v^{\ast}\| <\delta \textrm{ and }\Big\},
\end{align*}
we also have
\begin{align}\label{eq:betasec}
\langle \beta_1(u,v,t)-\beta_2(u,v,t), g_1 \rangle  \neq 0.
\end{align}
Moreover, there exist some $(u,v) \in O_{\delta}$ and some $k \in \mathbb{N}$ such that
\begin{equation}\label{eq:cos0}
\begin{split}
&\cos(\alpha_1(u,v,y_k,t)-\alpha_2(u,v,y_k,t)+\langle \beta_1(u,v,t)-\beta_2(u,v,t), \Pi_{ {\langle \Re\, \mathcal{U} \rangle}^\perp} x_k \rangle)\\
&\quad = \cos\Bigg(\arg \left(\frac{\Phi(t,u+v)}{\Phi(t,0)}\right)-\arg \left(\frac{\Phi(t,u)\Phi(t,v)}{\Phi(t,0)^2}\right)\\
&\quad \quad \quad+\langle \Im (\Pi_{\langle \Re \, \mathcal{U} \rangle}(\psi(t,u+v)-\Im \, \psi(t,u) -\Im \, \psi(t,v))),y_k \rangle \\
&\quad \quad \quad + \langle \beta_1(u,v,t)-\beta_2(u,v,t), \Pi_{ {\langle \Re\, \mathcal{U} \rangle}^\perp} x_k \rangle \Bigg)\\
&\quad\leq \frac{1}{3},
\end{split}
\end{equation}
since $ y_k $ stays in a bounded set and $ \Pi_{{\langle \Re \, \mathcal{U} \rangle}^\perp} x_k $ 
explodes with highest divergence rate in direction $g_1$.

However, inequality~\eqref{eq:cos0} now implies that 
\begin{multline*}
 r_1(u,v,y_k,t)r_2(u,v,y_k,t)\\
 \times(1-\cos(\alpha_1(u,v,y_k,t)-\alpha_2(u,v,y_k,t)+\langle \beta_1(u,v,t)-\beta_2(u,v,t), \Pi_{ {\langle \Re\, \mathcal{U} \rangle}^\perp} x_k \rangle)) 
> \frac{1}{2},
\end{multline*}
which contradicts~\eqref{eq:contra}. Since $g_1$ corresponds to the direction of the highest divergence rate, 
we thus conclude that
\[
\langle \beta_1(u,v,t)-\beta_2(u,v,t) ,g_1 \rangle =\Im (\langle \psi(t,u+v)-\psi(t,u)-\psi(t,v), g_1 \rangle) =0
\]  
for all $u,v \in B_{R(t)}$. Continuity of $u \mapsto \psi(t,u)$ therefore implies that 
$ u \mapsto \langle \psi(t,u), g_1 \rangle$ is a linear function.
Hence there exists a continuous curve of (real) vectors $ \lambda_1(t) \in V $ such that
\[
\langle \psi(t,u), g_1 \rangle =\langle \lambda_1(t),u\rangle
\] 
for all $u \in B_{R(t)}$.

We can now proceed inductively for the remaining directions of divergence $g_i$. 
Indeed, assume that $\langle \beta_1(u,v,t)-\beta_2(u,v,t) ,g_i \rangle=0$ for all $i \leq r-1$ and  all $u,v \in B_{R(t)}$. Then repeating the above steps allows us to conclude that $\langle \beta_1(u,v,t)-\beta_2(u,v,t) ,g_r \rangle=0$ 
for all $u,v \in B_{R(t)}$, yielding the assertion.
\end{proof}

As introduced before, we denote by $p$ the dimension of $\Re\, \mathcal{U}$. Let now $m^{\ast}$ be fixed such that $\dim (\Re\, \mathcal{U}^{m^{\ast}})=p$. For some $r > 0$, we define $K$ to be the intersection of $\mathcal{V}^{m^{\ast}}$ with the closed ball with center $0$ and radius $r$ in $\mathcal{U}$, that is,
\begin{align}\label{eq:ball}
K:=\overline{B}(0,r) \cap \mathcal{V}^{m^{\ast}}:=\{u \in \mathcal{U}\,|\, {\|\Re u\|}^2+{\| \Im u\|}^2 \leq r^2\} \cap \mathcal{V}^{m^{\ast}},
\end{align}
where $\mathcal{V}^{m^{\ast}}$ is defined in~\eqref{eq:mathcalVm}.
Let now $(u_1, \ldots, u_{p})$ be linearly independent vectors in $K \cap \Re \, \mathcal{U} $ and 
let $(u_{p+1},\ldots, u_n)$ be linearly independent vectors in $\Pi_{ {\langle \Re\, \mathcal{U} \rangle}^\perp}K$.
Then, as a consequence of the fact that $\psi(0,u)=u$ for all $u \in \mathcal{U} \supset K$ and the continuity of $t \mapsto \psi(t,u)$, there exists some $\delta>0$ such that for every $t \in [0, \delta)$
\[
(\psi(t, u_1), \ldots, \psi(t, u_p))
\]
and 
\[
(\Pi_{ {\langle \Re\, \mathcal{U} \rangle}^\perp}\psi(t,u_{p+1}), \ldots, \Pi_{ {\langle \Re\, \mathcal{U} \rangle}^\perp} \psi(t,u_{n}))
\]
are linearly independent.

Moreover, since $(t,u)\mapsto \Phi(t,u)e^{\langle \psi(t,u), x\rangle }$ is jointly 
continuous on $\re_+ \times \mathcal{U}^{m^{\ast}}$, with $\Phi(0,u)=1$ and $\psi(0,u)=u$ (see Proposition~\ref{prop:PhiPsiproperties}), it follows that there exists some $\eta >0$ such that for all $t \in[0, \eta]$
\begin{align}\label{eq:phipsibounded}
\inf_{u \in K}|\Phi(t,u)|>c \textrm{ and } \sup_{u \in K} (\|\Re \, \psi(t, u)\|^2 + \|\Im \, \psi(t, u)\|^2) < C ,
\end{align}
with some positive constants $c$ and $C$. By fixing these constants and some linearly independent vectors in $K$ as described above, we define
\begin{align}\label{eq:varepsilon}
\varepsilon:=\min( \eta, \delta).
\end{align}

Furthermore, let $ t\geq 0$ be fixed. Then we denote by $I^{\mathcal{T}}_{t,\varepsilon}$ the set
\begin{align}\label{eq:setI}
I^{\mathcal{T}}_{t,\varepsilon}:=(t,t+\varepsilon) \cap \mathcal{T},
\end{align}
where $\mathcal{T}$ is defined in~\eqref{eq:mathcalT}.

\begin{proposition}\label{prop:Xcadlag}
Let $K$ and $I^{\mathcal{T}}_{t,\varepsilon}$ be the sets defined in~\eqref{eq:ball} and~\eqref{eq:setI}. Consider the function
$\psi$ given in~\eqref{eq:affineprocess} with the properties of Proposition~\ref{prop:PhiPsiproperties}. 
Let $t \geq 0$  be fixed and consider a sequence 
$(q_k)_{k \in \mathbb{N}}$ with values in $\mathbb{Q}_+$ 
such that $q_k \uparrow t$. Moreover, 
let $(x_{q_k})_{k \in \mathbb{N}}$ be a sequence with values 
in $D_{\Delta} \cup \{\infty \}$.\footnote{As mentioned at the beginning of Section~\ref{sec:affgen}, $\infty$ corresponds to a ``point at infinity'' and $D_{\Delta} \cup \{\infty\}$ 
is the one-point compactification of $D_{\Delta}$. If the state space $D$ is compact, we \emph{do not} adjoin $\{\infty\}$ and
only consider a sequence with values in $D_{\Delta}$.}
Then the following assertions hold:
\begin{enumerate}
\item \label{item:Nneq0}
If for all 
$(T,u) \in I^{\mathcal{T}}_{t,\varepsilon} \times K$
\begin{align}\label{eq:Mfinitelimit}
\lim_{k\to \infty}N_{q_k}^{T,u}:= \lim_{k\to \infty}e^{\langle\psi(T-q_k,u),x_{q_k}\rangle}
\end{align}
exists finitely valued and does not vanish, 
then also $\lim_{k\to \infty}x_{q_k}$ exists finitely valued.
\item \label{item:Neq0}
If there exist some $(T,u) \in I^{\mathcal{T}}_{t,\varepsilon} \times K$ such that
\begin{align*}
\lim_{k\to \infty}N_{q_k}^{T,u}:= \lim_{k\to \infty}e^{\langle\psi(T-q_k,u),x_{q_k}\rangle}=0,
\end{align*}
then we have $\lim_{k \to \infty}\|x_{q_k}\|=\infty$. 
\end{enumerate}
Moreover, let $(q^T_k)_{k \in \mathbb{N}, T \in  I^{\mathcal{T}}_{t,\varepsilon}}$ be a family of sequences
with values in $\mathbb{Q}_+ \cap [t,T]$ such that $q^T_k \downarrow t$ for every $T \in I^{\mathcal{T}}_{t,\varepsilon}$ and 
the additional property that for every $S,T \in I^{\mathcal{T}}_{t,\varepsilon}$, with $S < T$, there exists some index $N \in \mathbb{N}$
such that, for all $k \geq N$, $q^S_{k-N}=q^T_k$. Then the above assertions hold true for these right limits 
with $q_k$ replaced by $q_k^T$.
\end{proposition}

\begin{remark}
Concerning assertion (ii) of Proposition~\ref{prop:Xcadlag}, note that, e.g.~in the case $q_k \uparrow t$, 
$\lim_{k \to \infty}\|x_{q_k}\|=\infty$ corresponds either to explosion 
or to the possibility that there exists some index $N \in \mathbb{N}$ such that $x_{q_k}=\Delta$ for all $k \geq N$. 
In the latter case we also have, due to the convention $\|\Delta\|=\infty$, $\lim_{k \to \infty}\|x_{q_k}\|=\infty$.
\end{remark}

\begin{proof}
We start by proving the first assertion~\ref{item:Nneq0}.
Let $T \in  I^{\mathcal{T}}_{t,\varepsilon}$ be fixed and define for all $u \in K$
\[
A(u):=\limsup_{k \to \infty} \left\langle \Re \, \psi(T-q_k,u), x_{q_k}\right\rangle, \quad a(u):=\liminf_{k \to \infty}  \left\langle \Re \, \psi(T-q_k,u), x_{q_k}\right\rangle.
\]
Then there exist subsequences $(x_{q_{k_{m}}})$ and $(x_{q_{k_{l}}})$ such that\footnote{Note that these subsequences depend on $u$. For notational convenience we however suppress the dependence on $u$.}
\begin{align*}
A(u)&=\lim_{m \to \infty} \left\langle \Re \, \psi(T-q_{k_{m}},u), x_{q_{k_{m}}}\right\rangle,\\ 
a(u)&=\lim_{l \to \infty} \left\langle \Re \, \psi(T-q_{k_{l}},u), x_{q_{k_{l}}}\right\rangle.
\end{align*}
First note that $A(u)$ and $a(u)$ exist finitely valued for all $u \in K$. Indeed, if there is some $u \in K$ such that $A(u)=\pm \infty$ or $a(u)=\pm \infty$, then the limit of $N_{q_{k}}^{T,u}$ does not exist, or $\lim_{k\to \infty}N_{q_{k}}^{T,u}$ is either $0$ or $+\infty$, 
which contradicts assumption~\eqref{eq:Mfinitelimit}. We now define
\begin{align*}
r_1(u)&=\lim_{m \to \infty} \exp\left(\left\langle \Re \, \psi(T-q_{k_{m}},u),x_{q_{k_{m}}}\right\rangle\right), \\ 
r_2(u)&=\lim_{l \to \infty} \exp\left(\left\langle \Re \, \psi(T-q_{k_{l}},u),x_{q_{k_{l}}}\right\rangle\right),\\
\varphi_m(u)&= \left\langle \Im \, \psi(T-q_{k_{m}},u),x_{q_{k_{m}}}\right\rangle, \\
\varphi_l(u)&= \left\langle \Im \, \psi(T-q_{k_{l}},u),x_{q_{k_{l}}}\right\rangle.
\end{align*}
Then the limits of $\cos(\varphi_m(u))$, $\cos(\varphi_l(u))$, $\sin(\varphi_m(u))$ and $\sin(\varphi_l(u))$ necessarily exist and
\begin{align*}
r_1(u) \lim _{m\to \infty}\cos (\varphi_m(u))&=r_2(u) \lim_{l \to \infty}\cos(\varphi_l(u)),\\
r_1(u) \lim _{m \to \infty} \sin (\varphi_m(u))&=r_2(u) \lim _{l \to \infty}\sin( \varphi_l(u)).
\end{align*}
This yields $r_1(u)=r_2(u)$ for all $u \in K$, since 
\[
\lim_{m \to \infty}\left( \cos^2(\varphi_m(u))+\sin^2 (\varphi_m(u)\right)=\lim_{l \to \infty}\left( \cos^2(\varphi_l(u))+\sin^2 (\varphi_l(u)\right)=1.
\]
In particular, we have proved that 
\begin{align}\label{eq:limrealpart}
\lim_{k \to \infty} \left\langle \Re \, \psi(T-q_k,u), x_{q_k}\right\rangle
\end{align}
exists finitely valued and does not vanish for all $(T,u) \in I^{\mathcal{T}}_{t,\varepsilon}\times K$. Choosing linear independent vectors $(u_1,\ldots, u_p) \in K \cap \Re \, \mathcal{U}$ thus implies that 
\[
\lim_{k \to \infty}\Pi_{\langle \Re \, \mathcal{U} \rangle} x_{q_k}
\]
exists finitely valued.

Therefore it only remains to focus on $ \Pi_{ {\langle \Re\, \mathcal{U} \rangle}^\perp} x_{q_k}$. From the above, we know in particular that for all $(T,u) \in I^{\mathcal{T}}_{t,\varepsilon}\times K$
\begin{align}\label{eq:limitIm}
\lim_{k \to \infty} e^{\left\langle \Pi_{ {\langle \Re\, \mathcal{U} \rangle}^\perp}\psi(T-q_k,u),  \Pi_{ {\langle \Re\, \mathcal{U} \rangle}^\perp}x_{q_k}\right\rangle}
\end{align}
exists finitely valued and does not vanish. This implies that for all $(T,u) \in I^{\mathcal{T}}_{t,\varepsilon}\times K$
\begin{align}\label{eq:Impsi}
\Im \left \langle \Pi_{ {\langle \Re\, \mathcal{U} \rangle}^\perp}\psi(T-q_k,u) ,\Pi_{ {\langle \Re\, \mathcal{U} \rangle}^\perp}x_{q_k}\right \rangle=\alpha_k(T,u)+2 \pi z_k(T,u),
\end{align}
where $\alpha_k(T,u) \in [-\pi, \pi)$, $\alpha(T,u):=\lim_{k \to \infty} \alpha_k(T,u)$ 
exists finitely valued and $(z_k(T,u))_{k \in \mathbb{N}}$ is a sequence with values in $\mathbb{Z}$, 
which a priori does not necessarily have a limit and\slash or $\lim_{k \to \infty} z_k(T,u)=\pm \infty$. 

Let us first assume that 
\begin{align}\label{eq:assumptioncontra}
\limsup_{k \to \infty} \|\Pi_{{\langle \Re\, \mathcal{U} \rangle}^\perp}x_{q_k}\|=\infty.
\end{align}
Then we are exactly in the situation of Lemma~\ref{lem:psilin} and the above limit~\eqref{eq:limitIm} can be written as  
\begin{align*}
\lim_{k \to \infty} e^{\left(\sum_i \langle  \lambda_i(T-q_k),u\rangle \, \langle g_i, x_{q_k}\rangle + \langle \Pi_{{\langle \Re\, \mathcal{U} \rangle}^\perp}\psi(T-q_k,u) , x_{q_k} - \sum_i g_i \, \langle g_i, x_{q_k} \rangle \rangle \right) }
\end{align*}
for all $u \in  \Pi_{{\langle \Re\, \mathcal{U} \rangle}^\perp}K$ with $\|\Im u\| < P(T)$, where $P(T)$ is 
defined by $P(T):=\inf_{k} R(T-q_k)$ and $R$ and the directions $g_i$ are given in Lemma~\ref{lem:psilin} after possibly selecting a subsequence such that
$ x_{q_k} - \sum_i g_i \, \langle g_i, x_{q_k} \rangle $ converges as $ k \to \infty $. 
Note that due to the strict positivity and continuity of $R$, $P(T)$ is strictly positive as well. 
Furthermore, there exists some $T^{\ast} \in I^{\mathcal{T}}_{t,\varepsilon}$ and some set 
$M_{T^{\ast}} \subseteq \{u \in \Pi_{{\langle \Re\, \mathcal{U} \rangle}^\perp}K\, |\, \|\Im u \| < P(T^{\ast}), \; 
\exists\, i \textrm{ s.t. } \langle \lambda_i (T^{\ast}-t),u \rangle \neq 0 \}$ of positive finite measure such that 
\begin{align}\label{eq:RiemLeb}
\lim_{k \to \infty} \int_{M_T^{\ast}} e^{\left \langle \Pi_{{\langle \Re\, \mathcal{U} \rangle}^\perp}\psi(T^{\ast}-q_k,u) , x_{q_k} - \sum_i g_i \, \langle g_i, x_{q_k} \rangle \right \rangle  } 
e^{(\sum_i \langle  \lambda_i (T^{\ast}-q_k),u \rangle \, \langle g_i, x_{q_k}\rangle)}du \neq 0.
\end{align}
However, it follows from the Riemann-Lebesgue Lemma that the previous limit is zero, whence contradicting~\eqref{eq:RiemLeb}. We therefore conclude that 
\[
\limsup_{k \to \infty} \|\Pi_{ {\langle \Re\, \mathcal{U} \rangle}^\perp}x_{q_k}\|<\infty.
\]
This in turn implies that there exists some $(T^{\ast},u^{\ast}) \in I_{t, \varepsilon}^{\mathcal{T}} \times K$ and $N \in \mathbb{N}$ such that for all $k \geq N$
\[
\Im \left \langle \Pi_{ {\langle \Re\, \mathcal{U} \rangle}^\perp} \psi(T^{\ast}-q_k, u^{\ast}), \Pi_{ {\langle \Re\, \mathcal{U} \rangle}^\perp} x_{q_k}\right \rangle \in  (-\pi, \pi).
\]
Indeed, this follows from the fact that for every $u \in K $ and $\eta >0$ there exists some $T^{\ast} \in I_{t, \varepsilon}^{\mathcal{T}}$ and $N \in \mathbb{N}$ such that for all $k \geq N$
\begin{align}\label{eq:beta}
\|\Im(\Pi_{ {\langle \Re\, \mathcal{U} \rangle}^\perp} \psi(T^{\ast}-q_k, u)-\Pi_{ {\langle \Re\, \mathcal{U} \rangle}^\perp} u)\| \leq \eta.
\end{align}
For $u^{\ast}$ with $\|\Im(\Pi_{ {\langle \Re\, \mathcal{U} \rangle}^\perp} u^{\ast})\|$ sufficiently small and $k$ sufficiently large, we thus have
\begin{align*}
&\left|\left \langle \Pi_{ {\langle \Re\, \mathcal{U} \rangle}^\perp} \Psi(T^{\ast}-q_k, u^{\ast}), \Pi_{ {\langle \Re\, \mathcal{U} \rangle}^\perp}x_{q_k}\right \rangle\right|\\
&\quad \leq (\|\Im( \Pi_{ {\langle \Re\, \mathcal{U} \rangle}^\perp} u^{\ast})\| +\|\Im(\Pi_{ {\langle \Re\, \mathcal{U} \rangle}^\perp} \Psi(T^{\ast}-q_k, u^{\ast})-\Pi_{ {\langle \Re\, \mathcal{U} \rangle}^\perp} u^{\ast})\|)\\
&\quad \quad \times (\limsup_{k \to \infty} \|\Pi_{ {\langle \Re\, \mathcal{U} \rangle}^\perp}x_{q_k}\|+1)\\
&\quad < \pi.
\end{align*}
Hence,
\begin{align}\label{eq:limit}
\lim_{k \to \infty}\Im \left\langle \Pi_{ {\langle \Re\, \mathcal{U} \rangle}^\perp}\psi(T^{\ast}-q_k, u^{\ast}),\Pi_{ {\langle \Re\, \mathcal{U} \rangle}^\perp}x_{q_k}\right\rangle=\alpha(T^{\ast},u^{\ast}).
\end{align}
As we can find $n-p$ linear independent vectors $u_{p+1}, \ldots, u_n$ such that~\eqref{eq:limit} is satisfied, we conclude that
\[
\lim_{k \to \infty} \Pi_{ {\langle \Re\, \mathcal{U} \rangle}^\perp}x_{q_k}
\]
exists finitely valued. This proves assertion~\ref{item:Nneq0}.

Concerning the second statement, observe that we have 
\begin{align}\label{eq:lim0M}
 \lim_{k\to \infty} e^{\langle\psi(T-q_k,u),x_{q_k}\rangle }=0,
\end{align}
if either explosion occurs or if $x_{q_N}$ jumps to $\Delta$ for some $N \in \mathbb{N}$ and $x_{q_k}=\Delta$ for
all $k \geq N$. (This happens when the corresponding process is killed.) Indeed, since~\eqref{eq:lim0M} is equivalent to 
$\lim_{k\to \infty} e^{\langle\Re \, \psi(T-q_k,u),x_{q_k}\rangle }=0$ and as $\psi(T-t,u)$ is bounded on $K$
due to the definition of $I^{\mathcal{T}}_{t,\varepsilon}$, we necessarily have
\[
\lim_{k\to \infty} \|x_{q_k}\|=\infty.
\]
In the case of a jump to $\Delta$, this is implied by the conventions $\|\Delta\|=\infty$ and
$f(\Delta)=0$ for any other function.

Similar arguments yield the assertion concerning right limits.
\end{proof}

Using Proposition~\ref{prop:Mcadlag} and Proposition~\ref{prop:Xcadlag} above, we are now prepared to prove Theorem~\ref{th:cadlagversion} below, 
which asserts the existence of a c\`adl\`ag version of $X$. Before stating this result, let us recall the notion of the (usual) augmentation 
of $(\mathcal{F}_t^0)$ with respect to $\mathbb{P}_x$, which guarantees the c\`adl\`ag version to be adapted.

\begin{definition}[Usual augmentation]\label{def:augmentation}
We denote by $\mathcal{F}^x$ the \emph{completion} of $\mathcal{F}^0$ with respect to $\mathbb{P}_x$.
A sub-$\sigma$-algebra $\mathcal{G} \subset \mathcal{F}^x$ is called \emph{augmented} with respect to $\mathbb{P}_x$ if
$\mathcal{G}$ contains all $\mathbb{P}_x$-null-sets in $\mathcal{F}^x$. The augmentation of $\mathcal{F}_t^0$ with respect
to $\mathbb{P}_x$ is denoted by $\mathcal{F}_t^x$, that is, $\mathcal{F}_t^x=\sigma(\mathcal{F}_t^0, \mathcal{N}(\mathcal{F}^x))$,
where $\mathcal{N}(\mathcal{F}^x)$ denotes all $\mathbb{P}_x$-null-sets in $\mathcal{F}^x$.
\end{definition}

\begin{theorem}\label{th:cadlagversion}
Let $X$ be an affine process relative to $(\mathcal{F}_t^0)$. Then there exists a process $\widetilde{X}$ such that, 
for each $x \in D_\Delta$, $\widetilde{X}$ is a $\mathbb{P}_x$-version of $X$,
which is c\`adl\`ag in $D_{\Delta} \cup \{\infty\}$ (in $D_{\Delta}$ respectively if $D$ is compact) and an affine process relative to $(\mathcal{F}_t^x)$.
As before, $\infty$ corresponds to a ``point at infinity'' and $D_{\Delta} \cup \{\infty\}$ 
is the one-point compactification of $D_{\Delta}$, if $D$ is non-compact.
\end{theorem}

\begin{remark}\label{rem:explosion}
We here establish the existence of a c\`adl\`ag version $\widetilde{X}$ 
whose sample paths may take $\infty$ as left limiting value if $D$ is non-compact. A priori, we cannot identify $\widetilde{X}_{s-}(\omega)$ with $\Delta$, whenever $\|\widetilde{X}_{s-}(\omega)\|=\infty$. Indeed, $\widetilde{X}_t(\omega)$ might become finitely valued for some $t \geq s$. 
This issue is clarified in Theorem~\ref{th:Xcadlaginfinity} below,
where we prove that $\mathbb{P}_x$-a.s.~$\|\widetilde{X}_t\| =\infty$ for all $t \geq s$ and all $s > 0$ if $\|\widetilde{X}_{s-}\|=\infty$. 
In particular, this allows us to identify $\infty$ with $\Delta$.

In the case $\widetilde{X}_{s}=\Delta$, which happens when the process is killed, 
Assumption~\eqref{eq:deltaabsorbing} guarantees that $\widetilde{X}_t =\Delta$ for all $t \geq s$ and all $s > 0$.
\end{remark}

\begin{proof}
It follows from Proposition~\ref{prop:Mcadlag} that for every $\omega \in \widetilde{\Omega}$\footnote{Note that due to the measurable projection theorem, $\widetilde{\Omega} \in \mathcal{F}^x$.}, where $\mathbb{P}_x[\widetilde{\Omega}]=1$, 
\[
t \mapsto M_t^{T,u}(\omega):=\Phi(T-t,u)e^{\langle\psi(T-t,u),X_t(\omega)\rangle}, \quad t \in [0,T],
\]
is the restriction to $ \mathbb{Q}_+ \cap [0,T]$ of a c\`adl\`ag function for all $(T,u) \in \mathcal{T} \times  \mathcal{V}$.
Here, $\widetilde{\Omega}$, $\mathcal{T}$ and $\mathcal{V}$ are defined in~\eqref{eq:omegatilde},~\eqref{eq:mathcalT} and~\eqref{eq:mathcalV}.
Hence, for every $\omega \in \widetilde{\Omega}$ and all $(T,u) \in \mathcal{T} \times  \mathcal{V}$, the limits
\begin{align}\label{eq:Mlimits}
\lim_{{\substack{q \in \mathbb{Q}_+\\ q \uparrow t}}}M_q^{T,u}(\omega), \quad \lim_{{\substack{q \in \mathbb{Q}_+\\ q \downarrow t}}}M_q^{T,u}(\omega)
\end{align}
exist finitely valued for all $t \in [0,T]$.  

Let us now show that the same holds true for $X$. For notational convenience we first focus on left limits.
Consider the sets $K$ and $I^{\mathcal{T}}_{t,\varepsilon}$ 
defined in~\eqref{eq:ball} and~\eqref{eq:setI} and let $t \geq 0$ be fixed. 
Take some sequence $(q_k)_{k \in \mathbb{N}}$, as specified in Proposition~\ref{prop:Xcadlag}, such that $q_k \uparrow t$. Then there exists some $N \in \mathbb{N}$ such that, for all $k \geq N$ and $(T,u) \in I^{\mathcal{T}}_{t,\varepsilon} \times K$, $\Phi(T-q_k,u) \neq 0$. This is a consequence of the definition of $\varepsilon$ (see~\eqref{eq:varepsilon}). Thus we can divide $M_{q_k}^{T,u}(\omega)$ by $\Phi(T-q_k,u)$ for all $k \geq N$ and 
$(T,u) \in I^{\mathcal{T}}_{t,\varepsilon} \times K$. By the continuity of $t \mapsto \Phi(t,u)$ and~\eqref{eq:Mlimits}, 
it follows that, for every $\omega \in \widetilde{\Omega}$, the limit
\[
\lim_{k \to \infty}N_{q_k}^{T,u}(\omega):=\lim_{k \to \infty}e^{\langle\psi(T-q_k,u),X_{q_k}(\omega)\rangle} 
\]
exists finitely valued for all $(T,u) \in  I^{\mathcal{T}}_{t,\varepsilon}\times K$. 
From Proposition~\ref{prop:Xcadlag} we thus deduce that, for every $\omega \in \widetilde{\Omega}$, the limit
\[
\lim_{k \to \infty}X_{q_k}(\omega)
\]
exists either finitely valued or $\lim_{k \to \infty}\|X_{q_k}(\omega)\|=\infty$. 
Using similar arguments yields the same assertion for right limits. 
Hence we can conclude that $\mathbb{P}_x$-a.s.
\begin{align}\label{eq:Xcadlag}
\widetilde{X}_t=\lim_{\substack{q \in \mathbb{Q}_+\\ q \downarrow t}}X_q
\end{align}
exists for all $t \geq 0$ and defines a c\`adl\`ag function in $t$. 

Let now $\Omega_0$ be the set of $\omega \in \Omega$ for which the limit $\widetilde{X}_t(\omega)$
exists for every $t$ and defines a c\`adl\`ag function in $t$. 
Then, as a consequence of~\citet[Theorem II.62.7, Corollary II.62.12]{rogers}, $\Omega_0 \in \mathcal{F}^0$
and $\mathbb{P}_x[\Omega_0]=1$ for all $x \in D_{\Delta}$. For $\omega \in \Omega\setminus \Omega_0$, we set $\widetilde{X}_t(\omega)=\Delta$ for all $t$. Then $\widetilde{X}$ is a c\`adl\`ag process and $\widetilde{X}_t$ is $\mathcal{F}^0$-measurable for every $t \geq 0$.
Since $X$ is assumed to be stochastically continuous, we have $X_s \to X_t$ in probability as $s \to t$. 
Using the fact that convergence in probability implies almost sure convergence along a subsequence, we have
\begin{align}\label{eq:limXt}
\mathbb{P}_x\left[\lim_{\substack{q \in \mathbb{Q}_+\\ q \downarrow t}} X_{q}=X_t\right]=1.
\end{align}
By our definition of $\widetilde{X}_t$, the limit in~\eqref{eq:limXt} is equal to $\widetilde{X}_t$ on $\Omega_0$. 
Hence, for all $x \in D_{\Delta}$, we have
$\mathbb{P}_x[\widetilde{X}_t=X_t]$ for each $t$, implying that $\widetilde{X}$ is a version of $X$. This then also yields
\[
\mathbb{E}_x\left[e^{\langle u, \widetilde{X}_t\rangle}\right]= \mathbb{E}_x\left[e^{\langle u, X_t\rangle}\right]
\]
and augmentation of $(\mathcal{F}_t^0)$ with respect to $\mathbb{P}_x$ ensures that $\widetilde{X}_t \in \mathcal{F}_t^x$ for each $t$.
We therefore conclude that $\widetilde{X}$ is an affine process with respect to $(\mathcal{F}_t^x)$.
\end{proof}

If $D$ is non-compact, the c\`adl\`ag version~\eqref{eq:Xcadlag} on $D_{\Delta} \cup \{\infty\}$, still denoted by $X$, can be realized on the space $\Omega':=\mathbb{D}'(D_{\Delta} \cup \{\infty\})$ of c\`adl\`ag paths 
$\omega:\re_+\to D_{\Delta} \cup \{\infty\}$ with $\omega(t)=\Delta$ for $t\geq s$, whenever $\omega(s)=\Delta$. 
However, we still have to prove that we can identify $\infty $ with $\Delta$, as mentioned in Remark~\ref{rem:explosion}.  
In other words, we have to show that $\|\omega(t)\|=\infty$ for all $t\geq s$ if explosion occurs for some $s > 0$, that is, $\|\omega(s-)\|=\infty$. This is stated in the Theorem~\ref{th:Xcadlaginfinity} below.
For its proof let us introduce the following notations:

Due to the convention $\|\Delta\|=\infty$, we define the \emph{explosion time} by (see~\citet{filipoyor} for a similar definition)
\begin{align*}
\Texpl&:=\left\{
\begin{array}{rl}
T_{\Delta}, & \textrm{if } T_k' < T_{\Delta} \textrm{ for all } k,\\
\infty, &\textrm{if } T_k'=T_{\Delta} \textrm{ for some } k, 
\end{array}\right.
\end{align*}
where the stopping times $T_{\Delta}$ and $T_k'$ are 
given by
\begin{align*}
T_{\Delta}&:=\inf\{t>0\, |\, \|X_{t-}\|=\infty \textrm{ or } \|X_t\|=\infty \},\\
T_k'&:=\inf\{t \, |\, \| X_{t-}\|\geq k \textrm{ or } \| X_{t}\|\geq k\}, \quad k \geq 1.
\end{align*}

Moreover, we denote by $\relint(C)$ the \emph{relative interior} of a set $C$ defined by
\begin{align*}
\relint(C)= \{x \in C\,|\, \overline{B}(x,r) \cap \aff(C) \subseteq C \textrm{ for some } r>0\},
\end{align*}
where $\aff(C)$ denotes the affine hull of $C$.

\begin{lemma}\label{lem:nonemptyU}
Let $X$ be an affine process with c\`adl\`ag paths in $D_{\Delta} \cup \{\infty\}$ and let $x \in D$ be fixed. If 
\begin{align}\label{eq:posprobexp}
 \mathbb{P}_x[\Texpl< \infty]>0,
\end{align}
then $\relint(\Re \,\mathcal{U}) \neq \emptyset$ and we have $\mathbb{P}_x$-a.s.
\[
 \lim_{t \uparrow \Texpl}e^{\langle u, X_t \rangle}=0 
\]
for all $u \in \relint(\Re \,\mathcal{U})$. 
\end{lemma}

\begin{proof}
Let us first establish that under Assumption~\eqref{eq:posprobexp}, $\relint{\Re\, \mathcal{U}} \neq \emptyset$.
To this end, we denote by $\Omegaexpl$ the set
\[
\Omegaexpl=\{\omega \in \Omega' \, |\, \Texpl(\omega) < \infty\}.
\]
Then it follows from Proposition~\ref{prop:Mcadlag} and~\ref{prop:Xcadlag} that, for 
$\mathbb{P}_x$-almost every $\omega \in \Omegaexpl$, there exist some 
$(T(\omega),v(\omega)) \in (\Texpl(\omega), \infty) \times \im V$ such that 
\[
\lim_{t \uparrow \Texpl(\omega)}\Phi(T(\omega)-t,v(\omega))\neq 0
\] 
and
\begin{align}\label{eq:lim0}
 \lim_{t \uparrow \Texpl(\omega)} N_{t}^{T(\omega),v(\omega)}(\omega)= \lim_{t \uparrow \Texpl(\omega)}e^{\langle\psi(T(\omega)-t,v(\omega)),X_t(\omega)\rangle}=0.
\end{align}
This implies that 
\begin{align}\label{eq:limmininf}
\lim_{t \uparrow \Texpl(\omega)}\langle \Re \, \psi(T(\omega)-t,v(\omega)),X_{t}(\omega)\rangle=-\infty,
\end{align} 
and in particular that $\mathcal{U} \ni \Re \, \psi(T(\omega)-\Texpl(\omega),v(\omega))\neq 0$, 
which proves the claim, since $\Re \,\mathcal{U}\subseteq\overline{\relint(\Re \, \mathcal{U})}$.

Furthermore, by~\eqref{eq:limmininf} we have
$\lim_{t \uparrow \Texpl(\omega)}\| \Pi_{\langle\Re\, \mathcal{U}\rangle} (X_t(\omega))\|=\infty$
and an application of Lemma~\ref{lem:uint} below yields the second assertion.
\end{proof}

\begin{lemma}\label{lem:uint}
Assume that $\relint(\Re\, \mathcal{U}) \neq \emptyset$ and that there exists some $D$-valued sequence $(x_k)_{k \in \mathbb{N}}$ such that 
\begin{align}\label{eq:projsequence}
\lim_{k \to \infty} \|\Pi_{\langle \Re\, \mathcal{U} \rangle}x_k\|=\infty.
\end{align}
Then $\lim_{k \to \infty} \langle u, x_k\rangle =-\infty$ for all $u \in \relint(\Re\, \mathcal{U})$.
\end{lemma}

\begin{proof}
Suppose by contradiction that there exists some $u\in \relint(\Re\, \mathcal{U})$ such that 
\[
\limsup_{k \to \infty}\langle u, x_k\rangle >-\infty.
\]
Then there exists a subsequence, still denoted by $(x_k)$, such that 
\begin{align}\label{eq:contradiction}
\lim_{k \to \infty}\langle u, x_k\rangle >-\infty.
\end{align}
and due to~\eqref{eq:projsequence} some direction 
$g \in \langle \Re\, \mathcal{U} \rangle$ such that
\begin{align}\label{eq:directiondiv}
\lim_{k \to \infty}\langle g, x_k \rangle =\infty. 
\end{align}
Moreover, since $u \in \relint(\Re\, \mathcal{U})$, there exists some $\varepsilon >0 $ such that
$u +\varepsilon g \in \relint(\Re \, \mathcal{U})$. By the definition of $\mathcal{U}$, we have
\[
\sup_{x \in D} \langle u +\varepsilon g, x\rangle < \infty.
\]
Due to~\eqref{eq:directiondiv}, this however implies that 
\[
\lim_{k \to \infty}\langle u,x_k \rangle =-\infty
\]
and contradicts~\eqref{eq:contradiction}. 
\end{proof}

\begin{theorem}\label{th:Xcadlaginfinity}
Let $X$ be an affine process with c\`adl\`ag paths in $D_{\Delta}\cup \{\infty\}$. Then, for every $x \in D$, the following assertion holds $\mathbb{P}_x$-a.s.: If
\begin{align}\label{eq:assumption}
\|X_{s-}\|=\infty, 
\end{align}
then $\|X_t\| =\infty$ for all $t\geq s$ and $s \geq 0$. Identifying $\infty$ with $\Delta$, then yields
$X_t=\Delta$ for all $t \geq s$.
\end{theorem}

\begin{proof}
Let $x \in D$ be fixed and let $u \in \relint(\Re \,\mathcal{U})$. 
Note that by Lemma~\ref{lem:nonemptyU} $\relint(\Re \,\mathcal{U}) \neq \emptyset$
and that $\Phi(t,u)$ and $\psi(t,u)$ are real-valued functions with values in $\re_{++}$ and
$\Re \,\mathcal{U}$, respectively. 
Take now some $T>0$ and $\delta>0$ such that 
\[
 \mathbb{P}_x\left[ T-\delta < \Texpl \leq T \right]>0,
\]
and $\psi(t,u) \in \relint(\Re \,\mathcal{U})$ for all $t < \delta$. 
Consider the martingale
\[
 M_t^{T,u}=\Phi(T-t,u)e^{\langle \psi(T-t,u), X_t \rangle }, \quad t \leq T,
\]
which is clearly nonnegative and has c\`adl\`ag paths. Moreover, by the choice of $\delta$, it follows from 
Lemma~\ref{lem:nonemptyU} and the conventions $\| \Delta \| =\infty$ and $f(\Delta)=0$ for any other function that $\mathbb{P}_x$-a.s.
\begin{align}\label{eq:Z0}
 M_{s-}^{T,u}=0,  
\quad s \in (T-\delta, T], 
\end{align}
if and only if $\|X_{s-}\|=\infty$ 
for $s \in (T-\delta, T]$. 
We thus conclude using~\citet[Theorem II.78.1]{rogers} that $\mathbb{P}_x$-a.s.~$M_t^{T,u}=0$ for all $t \geq s$, which in turn 
implies that $\|X_{t}\|=\infty$ for all $t \geq s$. This allows us to identify $\infty$ with $\Delta$ and we obtain
$X_t=\Delta$ for all $t \geq s$. Since $T$ was chosen arbitrarily, the assertion follows.
\end{proof}

Combining Theorem~\ref{th:cadlagversion} and Theorem~\ref{th:Xcadlaginfinity} and using Assumption~\eqref{eq:deltaabsorbing}, we thus obtain the following statement:

\begin{corollary}\label{col:cadlag}
Let $X$ be an affine process relative to $(\mathcal{F}_t^0)$. Then there exists a process $\widetilde{X}$ such that, for each $x \in D_\Delta$, $\widetilde{X}$ is a $\mathbb{P}_x$-version of $X$, which 
is an affine process relative to $(\mathcal{F}_t^x)$, whose paths are c\`adl\`ag and satisfy $\mathbb{P}_x$-a.s.~$\widetilde{X}_t=\Delta$ for $t\geq s$, whenever $\|\widetilde{X}_{s-}\|=\infty$ or 
$\|\widetilde{X}_s\|=\infty$.
\end{corollary}

\begin{remark}\label{rem:Omegacadlag}
We will henceforth \emph{always} assume that we are using the c\`adl\`ag version
of an affine process, given in Corollary~\ref{col:cadlag}, which we still denote by $X$.
Under this assumption $X$ can now be realized on the space $\Omega=\mathbb{D}(D_{\Delta})$ of c\`adl\`ag paths 
$\omega:\re_+\to D_{\Delta}$ with $\omega(t)=\Delta$ for $t\geq s$, whenever $\|\omega(s-)\|=\infty$ or $\|\omega(s)\|=\infty$. 
The canonical realization of an affine process $X$ is then defined by $X_t(\omega)=\omega(t)$.
Moreover, we make the convention that $X_{\infty}=\Delta$, which allows us to write certain formulas without restriction.
\end{remark}

\section{Right-Continuity of the Filtration and Strong Markov Property}\label{sec:filt}

Using the existence of a right-continuous version of an affine process, we can now show that $(\mathcal{F}_t^x)$, that is, 
the augmentation of $(\mathcal{F}_t^0)$ with respect to $\mathbb{P}_x$, is right-continuous.

\begin{theorem}\label{th:rightcontfilt}
Let $x \in D$ be fixed and let $X$ be an affine process relative to $(\mathcal{F}_t^x)$ with c\`adl\`ag paths. 
Then $(\mathcal{F}_t^x)$ is right-continuous.
\end{theorem}

\begin{proof}
We adapt the proof of~\citet[Theorem I.4.31]{protter} to our setting. 
We have to show that for every $t \geq 0$, $\mathcal{F}^x_{t+}=\mathcal{F}_t^x$, where $\mathcal{F}^x_{t+}=\bigcap_{s >t}\mathcal{F}_t^x$. 
Since the filtration is increasing, it suffices to show that $\mathcal{F}^x_{t}=\bigcap_{n \geq 1}\mathcal{F}_{t+\frac{1}{n}}^x$.
In particular, we only need to prove that
\begin{align}\label{eq:rightcont}
\mathbb{E}_x\left[e^{\langle u_1, X_{t_1}\rangle +\cdots +\langle u_k, X_{t_k}\rangle}\,\Big |\, \mathcal{F}_t^x\right]=
\mathbb{E}_x\left[e^{\langle u_1, X_{t_1}\rangle  +\cdots +\langle u_k, X_{t_n}\rangle}\,\Big |\, \mathcal{F}_{t+}^x\right]
\end{align}
for all $(t_1, \ldots, t_k)$ and all $(u_1, \ldots, u_k)$ with $t_i \in \re_+$ and $u_i \in \mathcal{U}$, since this implies
$\mathbb{E}_x[Z |\mathcal{F}_t^x]=\mathbb{E}_x[Z | \mathcal{F}_{t+}^x]$ for every bounded $Z \in \mathcal{F}^x$. 
As both $\mathcal{F}_{t+}^x$ and $\mathcal{F}_{t}^x$ contain the nullsets $\mathcal{N}(\mathcal{F}^x)$, this then already yields
$\mathcal{F}_{t+}^x =\mathcal{F}_{t}^x$ for all $t \geq 0$.

In order to prove~\eqref{eq:rightcont}, let $t\geq 0$ be fixed and take first $t_1 \leq t_2 \cdots \leq t_k \leq t$.  
Then we have for all $(u_1,\ldots, u_k)$
\begin{align*}
\mathbb{E}_x\left[e^{\langle u_1, X_{t_1}\rangle +\cdots +\langle u_k, X_{t_k}\rangle}\,\Big |\, \mathcal{F}_t^x\right]&=
\mathbb{E}_x\left[e^{\langle u_1, X_{t_1}\rangle +\cdots +\langle u_k, X_{t_k}\rangle}\,\Big |\, \mathcal{F}_{t+}^x\right]\\
&=e^{\langle u_1, X_{t_1}\rangle +\cdots +\langle u_k, X_{t_k}\rangle}.
\end{align*}
In the case $t_k> t_{k-1} \cdots > t_1 > t$, we give the proof for $k=2$ for notational convenience. 
Let $t_2>t_1>t$ and fix $u_1, u_2 \in \mathcal{U}$. Then we have by the affine property
\begin{align*}
\mathbb{E}_x\left[ e^{\langle u_1, X_{t_1}\rangle +\langle u_2, X_{t_2}\rangle}\,\Big |\, \mathcal{F}_{t+}^x\right]&=\lim_{s\downarrow t}\mathbb{E}_x\left[ e^{\langle u_1, X_{t_1}\rangle +\langle u_2, X_{t_2}\rangle}\,\Big |\, \mathcal{F}_s^x\right]\\
&=\lim_{s\downarrow t}\mathbb{E}_x\left[ \mathbb{E}_x\left[e^{\langle u_1, X_{t_1}\rangle +\langle u_2, X_{t_2}\rangle}\,\Big |\, \mathcal{F}_{t_1}^x\right]\,\Big |\, \mathcal{F}_s^x\right]\\
&=\Phi(t_2-t_1,u_2)\lim_{s\downarrow t}\mathbb{E}_x\left[ e^{\langle u_1+\psi(t_2-t_1,u_2), X_{t_1}\rangle} \,\Big |\, \mathcal{F}_s^x\right].
\end{align*}
If $\Phi(t_2-t_1,u_2)=0$, it follows by the same step that 
\[
\mathbb{E}_x\left[ e^{\langle u_1, X_{t_1}\rangle +\langle u_2, X_{t_2}\rangle}\,\Big |\, \mathcal{F}_{t}^x\right]=0,
\] 
too. Otherwise, we have by Proposition~\ref{prop:PhiPsiproperties}~\ref{item:rangePsi}, $\psi(t_2-t_1,u_2) \in \mathcal{U}$, 
and by the definition of $\mathcal{U}$ also $u_1+\psi(t_2-t_1,u_2) \in \mathcal{U}$. 
Hence, again by the affine property and right-continuity of $t \mapsto X_t(\omega)$, the above becomes 
\begin{align*}
&\mathbb{E}_x\left[ e^{\langle u_1, X_{t_1}\rangle +\langle u_2, X_{t_2}\rangle}\,\Big |\, \mathcal{F}_{t+}^x\right]\\
&\quad=\Phi(t_2-t_1,u_2)\lim_{s\downarrow t}\Phi(t_1-s,u_1+\psi(t_2-t_1,u_2))e^{\langle \psi(t_1-s,u_1+\psi(t_2-t_1,u_2)), X_{s}\rangle}\\
&\quad=\Phi(t_2-t_1,u_2)\Phi(t_1-t,u_1+\psi(t_2-t_1,u_2))e^{\langle \psi(t_1-t,u_1+\psi(t_2-t_1,u_2)), X_{t}\rangle}\\
&\quad=\mathbb{E}_x\left[ e^{\langle u_1, X_{t_1}\rangle +\langle u_2, X_{t_2}\rangle} \,\Big |\, \mathcal{F}_{t}^x\right].
\end{align*}
This yields~\eqref{eq:rightcont} and by the above arguments we conclude that 
$\mathcal{F}_{t+}^x =\mathcal{F}_{t}^x$ for all $t \geq 0$.
\end{proof}

\begin{remark}\label{rem:setting}
A consequence of Theorem~\ref{th:rightcontfilt} is that $(\Omega, \mathcal{F}_t, (\mathcal{F}_t^x), \mathbb{P}_x)$ 
satisfies the \emph{usual conditions}, since
\begin{enumerate}
\item $\mathcal{F}^x$ is $\mathbb{P}_x$-complete, 
\item $\mathcal{F}_0^x$ contains all $\mathbb{P}_x$-null-sets in $\mathcal{F}^x$, 
\item $(\mathcal{F}^x_t)$ is right-continuous.
\end{enumerate}

Let us now set
\begin{align}\label{eq:standardfilt}
\mathcal{F}:=\bigcap_{x \in D_{\Delta}}\mathcal{F}^x, \quad \mathcal{F}_t:=\bigcap_{x \in D_{\Delta}}\mathcal{F}_t^x.
\end{align}
Then $(\Omega, \mathcal{F}, (\mathcal{F}_t), \mathbb{P}_x)$ does not necessarily satisfy the usual conditions, 
but $\mathcal{F}_t=\mathcal{F}_{t+}$ still holds true.
Moreover, it follows e.g.~from~\citet[Proposition III.2.12, III.2.14]{revuzyor} that, for each $t$, $X_t$ is $\mathcal{F}_t$-measurable
and a Markov process with respect to $(\mathcal{F}_t)$. 

Unless otherwise mentioned, we henceforth \emph{always} consider affine processes 
on the filtered space $(\Omega, \mathcal{F}, (\mathcal{F}_t))$, where  $\Omega=\mathbb{D}(D_{\Delta})$, 
as described in Remark~\ref{rem:Omegacadlag}, and $\mathcal F$, $\mathcal{F}_t$ are given by~\eqref{eq:standardfilt}. 
Notice that these assumptions on the probability space correspond to the standard setting considered for Feller
processes (compare~\citet[Definition III.7.16, III.9.2]{rogers}).
\end{remark}

Similar as in the case of Feller processes, we can now formulate and prove the strong Markov property 
for affine processes using the above setting and in particular the right-continuity of the sample paths.

\begin{theorem}
Let $X$ be an affine process and let $T$ be a $(\mathcal{F}_t)$-stopping time. 
Then for each bounded Borel measurable function $f$ and $s \geq 0$
\[
 \mathbb{E}_x\left[f(X_{T+s})|\mathcal{F}_T\right]= \mathbb{E}_{X_T}\left[f(X_s)\right], \quad \mathbb{P}_x \textrm{-a.s.}
\]
\end{theorem}

\begin{proof}
This result can be shown by the same arguments used to prove the strong Markov property of Feller processes
(see, e.g.,~\citet[Theorem 8.3, Theorem 9.4]{rogers}), namely by using a dyadic approximation of the stopping time $T$
and applying the Markov property. Instead of using $C_0$-functions and the Feller property, 
we here consider the family of functions $\{x \mapsto e^{\langle u, x \rangle}\,|\, u \in \im V\}$ 
and the affine property, which asserts in particular
that 
\[
 x \mapsto \mathbb{E}_x\left[e^{\langle u, X_t\rangle}\right]=P_te^{\langle u, x\rangle}=\Phi(t,u)e^{\langle \psi(t,u),x \rangle}
\]
is continuous. This together with the right-continuity of paths then implies for every $\Lambda \in \mathcal{F}_T$ and $u \in \im V$
\[
\mathbb{E}_x\left[e^{\langle u, X_{T+s}\rangle} 1_{\Lambda}\right]=\mathbb{E}_x\left[P_se^{\langle u, X_{T}\rangle} 1_{\Lambda}\right].
\]
The assertion then follows by the same arguments as in~\citet[Theorem 8.3]{rogers} or~\citet[Theorem 2.3.1]{chung}.
\end{proof}

\section{Semimartingale Property}\label{sec:semimart}

We shall now relate affine processes to semimartingales, where, for every $x \in D$, 
semimartingales are understood with respect to the 
filtered probability space $(\Omega, \mathcal{F}, (\mathcal{F}_t), \mathbb{P}_x)$ defined above. By convention, we call $X$ a semimartingale if 
$X1_{[0, T_{\Delta})}$ is a semimartingale, 
where -- as a consequence of Theorem~\ref{th:Xcadlaginfinity} and Corollary~\ref{col:cadlag} --
we can now define the \emph{lifetime} $T_{\Delta}$ by
\begin{align}\label{eq:lifetimeX}
T_{\Delta}(\omega)=\inf\{t>0\, |\,  X_t(\omega)=\Delta\}.
\end{align}

Let us start with the following definition for general Markov processes (compare~\citet[Definition 7.1]{cinlar}):

\begin{definition}[Extendend Generator]\label{def:extendedgen}
An operator $\mathcal{G}$ with domain $\mathcal{D}_{\mathcal{G}}$ is called \emph{extended generator} for a Markov process $X$ 
(relative to some filtration $(\mathcal{F}_t)$) if $D_{\mathcal{G}}$ consists of those Borel measurable functions $f: D \to \mathbb{C}$ for which there exists a function $\mathcal{G}f$ such that the process
\[
f(X_t)-f(x)-\int_0^t \mathcal{G}f(X_{s-})ds
\] 
is a well-defined and $(\mathcal{F}_t,\mathbb{P}_x)$-local martingale for every $x \in D_{\Delta}$.
\end{definition}

In the following lemma we consider a particular class of functions 
for which it is possible to state the form of the extended generator for a Markov process in terms of its semigroup.

\begin{lemma}\label{lem:gmart}
Let $X$ be a $D_{\Delta}$-valued Markov process relative to some filtration $(\mathcal{F}_t)$.
Suppose that $u \in \mathcal{U}$ 
and $\eta > 0$. Consider the function 
\[
g_{u,\eta}:D \to \cn, x \mapsto g_{u, \eta}(x):=\frac{1}{\eta}\int_0^{\eta}P_s e^{\langle u , x\rangle}ds.
\] 
Then, for every $x \in D$,
\begin{align*}
M_t^u:=g_{u,\eta}(X_t)-g_{u,\eta}(X_0)-\int_0^t \frac{1}{\eta}\left(P_{\eta}e^{\langle u, X_{s-}\rangle}-e^{\langle u,X_{s-}\rangle}\right) ds
\end{align*}
is a (complex-valued) $(\mathcal{F}_t,\mathbb{P}_x)$-martingale and thus $g_{u,\eta}(X)$ is a (complex-valued) special semimartingale.
\end{lemma}

\begin{proof}
Since $g_{u,\eta}$ and $P_{\eta}e^{\langle u, \cdot\rangle}-e^{\langle u,\cdot\rangle}$ are bounded, $M_t^u$ is integrable for each $t$ and we have
\begin{align*}
&\mathbb{E}_x\left[M_t^u |\mathcal{F}_r \right]\\
&\quad=M_r^u+\mathbb{E}_x\left[g_{u,\eta}(X_t)-g_{u,\eta}(X_r)-\int_r^t \frac{1}{\eta}\left(P_{\eta}e^{\langle u, X_{s-}\rangle}-e^{\langle u,X_{s-}\rangle}\right) ds\, \Big|\,\mathcal{F}_r \right]\\
&\quad=M_r^u+\mathbb{E}_{X_r}\left[g_{u,\eta}(X_{t-r})-g_{u,\eta}(X_0)-\int_0^{t-r} \frac{1}{\eta}\left(P_{\eta}e^{\langle u, X_{s-}\rangle}-e^{\langle u,X_{s-}\rangle}\right) ds \right]\\
&\quad=M_r^u+\frac{1}{\eta}\int_{t-r}^{t-r+\eta} P_se^{\langle u,X_{r}\rangle}ds-\frac{1}{\eta}\int_{0}^{\eta} P_se^{\langle u,X_{r}\rangle}ds\\
&\quad\quad -\frac{1}{\eta}\int_{\eta}^{t-r+\eta} P_se^{\langle u,X_{r}\rangle}ds+\frac{1}{\eta}\int_{0}^{t-r} P_se^{\langle u,X_{r}\rangle}ds\\
&\quad=M_r^u.
\end{align*}
Hence $M^u$ is $(\mathcal{F}_t,\mathbb{P}_x)$-martingale and thus $g_{u,\eta}(X)$ is a special semimartingale, since it is the sum of a martingale and a predictable finite variation process.
\end{proof}

\begin{remark}\label{rem:extendedgen}
Lemma~\ref{lem:gmart} asserts that the extended generator applied to $g_{u,\eta}$ is given by
$\mathcal{G}g_{u,\eta}(x)=\frac{1}{\eta}\left(P_{\eta}e^{\langle u, x\rangle}-e^{\langle u,x\rangle}\right)$.
Note that for general Markov processes and even for affine processes we do not know whether 
the ``pointwise'' infinitesimal generator applied to 
\[
 e^{\langle u,x \rangle}=\lim_{\eta \to 0}g_{u,\eta}=\lim_{\eta \to 0}\frac{1}{\eta}\int_0^{\eta}P_s e^{\langle u , x\rangle}ds,
\]
that is, 
\[
 \lim_{\eta \to 0} \frac{1}{\eta}\left(P_{\eta}e^{\langle u, x\rangle}-e^{\langle u,x\rangle}\right),
\]
is well-defined or not.\footnote{In the case of affine processes, this would be implied by the differentiability of
$\Phi$ and $\psi$ with respect to $t$, which we only prove in Section~\ref{sec:regularity} using the results of this paragraph.}
For this reason we consider the family of functions 
$\{ x \mapsto g_{u,\eta}(x) \, |\, u \in \mathcal{U},\, \eta > 0\}$, which exhibits in the case of affine processes similar properties as 
$\{ x \mapsto e^{\langle u, x\rangle} \, |\, u \in \mathcal{U}\}$ (see Remark~\ref{rem:fullcomplete}~(ii) and Lemma~\ref{lem:fullcomplete} below).
These properties are introduced in the following definitions (compare~\citet[Definition 7.7, 7.8]{cinlar}).
\end{remark}

\begin{definition}[Full Class]\label{def:full}
A class $\mathcal{C}$ of Borel measurable functions from $D$ to $\mathbb{C}$ is said to be a \emph{full} class if, for all $r \in \mathbb{N}$,
there exists a finite family $\{f_1, \ldots, f_N\} \in \mathcal{C}$ and a function $h \in C^2(\mathbb{C}^N,D)$ such that
\begin{align}\label{eq:fcth}
x=h(f_1(x), \ldots, f_N(x))
\end{align}
for all $x \in D$ with $\| x\| \leq r$.
\end{definition}

\begin{definition}[Complete Class]\label{def:completeclass}\label{def:complete}
Let $\beta \in V$, $\gamma \in S_+(V)$, where $S_+(V)$ denotes the positive semidefinite matrices over $V$, 
and let $F$ be a nonnegative measure on $V$, which integrates $(\|\xi\|^2 \wedge 1)$, satisfies $F(\{0\})=0$
and $x + \supp(F) \subseteq D_{\Delta}$ for all $x \in D$. Moreover, let $\chi: V \to V$ denote some truncation function, 
that is, $\chi$ is bounded and satisfies $\chi(\xi)=\xi$ in a neighborhood of $0$. 
A countable subset of functions $\mathcal{\widetilde{C}} \subset C_b^2(D)$ is called \emph{complete} if, for any fixed $x \in D$, the countable collection of numbers
\begin{multline}\label{eq:completeclass}
\kappa(f(x))=\langle \beta,\nabla f(x) \rangle + \frac{1}{2}\sum_{i,j} \gamma_{ij} D_{ij}f(x)\\ 
+ \int_V \left(f(x+\xi)-f(x)-\langle \nabla f(x),\chi(\xi) \rangle\right)F(d\xi), \quad f \in \mathcal{\widetilde{C}}
\end{multline}
completely determines $\beta$, $\gamma$ and $F$.
A class $\mathcal{C}$ of Borel measurable functions from $D$ to $\mathbb{C}$ is said to be \emph{complete class} if it contains such a countable set. 
\end{definition}

\begin{remark}\label{rem:fullcomplete}
\begin{enumerate}
\item Note that the integral in~\eqref{eq:completeclass} is well-defined for all $f \in C_b^2(D)$. This is a consequence of the integrability assumption and the fact that $x + \supp(F)$ is supposed to lie in $D_{\Delta}$ for all $x$.
\item The class of functions
\begin{align}\label{eq:CiV}
\mathcal{C}^{\ast}:=\left\{ D \to \mathbb{C},\, x \mapsto e^{\langle u, x\rangle}\, \big|\, u \in \im V \right\} 
\end{align}
is a full and complete class. Indeed, for every $x \in D$ with $\| x\| \leq r$, we can find $n$ linearly independent vectors $(u_1, \ldots, u_n)$ such that
\[
\Im\langle u_i, x\rangle \in \left[-\frac{\pi}{2},\frac{\pi}{2}\right].
\]
This implies that $x$ is given by
\[
x=\left(\arcsin \left(\Im e^{\langle u_1,x\rangle}\right), \ldots, \arcsin \left(\Im e^{\langle u_n,x\rangle}\right)\right)(\Im u_1, \ldots, \Im u_n)^{-1}
\]
and proves that $\mathcal{C^{\ast}}$ is a full class. Completeness follows by the same arguments as in~\citet[Lemma II.2.44]{jacod}.
\end{enumerate}
\end{remark}

\begin{lemma}\label{lem:fullcomplete}
Let $X$ be an affine process with $\Phi$ and $\psi$ given in~\eqref{eq:affineprocess}.
Consider the class of functions 
\begin{align}\label{eq:mathcalC}
\mathcal{C}:=\left\{ D \to \mathbb{C},\,  x \mapsto g_{u, \eta}(x):=\frac{1}{\eta}\int_0^{\eta}\Phi(s,u)e^{\langle \psi(s,u) , x\rangle}ds \, \big|\, u \in \im V, \, \eta >0 \right\}. 
\end{align}
Then $\mathcal{C}$ is a full and complete class.
\end{lemma}

\begin{proof}
Let $(u_1,\ldots u_n) \in \im V$ be $n$ linearly independent vectors and define a function $f_{\eta}: D\rightarrow \mathbb{C}^n$ by $f_{\eta,i}(x)= g_{u_i,\eta}(x)$. Then the Jacobi matrix $J_{f_{\eta}}(x)$ is given by
\begin{align*}
\small
\left(\begin{array}{ccc}
\frac{1}{\eta} \int_0^{\eta} \Phi(s,u_1)e^{\langle \psi(s,u_1), x\rangle }\psi_1(s,u_1)ds & \ldots & \frac{1}{\eta} \int_0^{\eta} \Phi(s,u_1)e^{\langle \psi(s,u_1), x\rangle }\psi_n(s,u_1)ds\\
 \vdots &\ddots& \vdots\\
 \frac{1}{\eta} \int_0^{\eta} \Phi(s,u_n)e^{\langle \psi(s,u_n), x\rangle }\psi_1(s,u_n)ds & \ldots & \frac{1}{\eta} \int_0^{\eta} \Phi(s,u_n)e^{\langle \psi(s,u_n), x\rangle }\psi_n(s,u_n)ds
\end{array}\right).
\end{align*}
In particular, the imaginary part of each row tends to $(\cos(\Im\langle u_i, x\rangle)\Im u_i)^{\top}$ for $\eta \to 0$. Hence there exists some $\eta >0$ such that the rows of $\Im J_{f_{\eta}}$ are linearly independent. As $\Im f_{\eta}: D \to \re^n$ is a $C^{\infty}(D)$-function and as $J_{\Im f_{\eta}}=\Im J_{f_{\eta}}$, it follows from the inverse function theorem that, for each $x_0 \in D$, there exists some $r_0 >0$ such that $\Im f_{\eta}: B(x_0,r_0) \to W$ has a $C^{\infty}(W)$ inverse, where $W= 
\Im f_{\eta}( B(x_0,r_0))$.

Let now $r \in \mathbb{N}$ and consider $x \in D$ with $\| x \| \leq r$.  Assume without loss of generality that $0 \in D$ and let $x_0=0$. Since
\[
\lim_{\eta \to 0}J_{\Im f_{\eta}}(x)=(\cos(\Im\langle u_1, x\rangle)\Im u_1, \ldots, \cos(\Im \langle u_n, x\rangle)\Im u_n)^{\top},
\]
we can assure -- by choosing the linearly independent vectors $(u_1, \ldots, u_n)$ such that $|\langle u_i, x\rangle|$ is small enough -- that for all $x \in \overline{B}(0, r) \cap D$
\begin{multline*}
\|\lim_{\eta \to 0}J^{-1}_{\Im f_{\eta}}(0)\lim_{\eta \to 0}J_{\Im f_{\eta}}(x)-I\|\\
=\|(\Im u_1, \ldots, \Im u_n)^{-\top}(\cos(\Im \langle u_1, x\rangle)\Im u_1, \ldots, \cos(\Im \langle u_n, x\rangle)\Im u_n)^{\top}-I\|<1.
\end{multline*}
By the continuity of the matrix inverse the same holds true for $\eta$ small enough.
The proof of the inverse function theorem (see, e.g.,~\citet[Theorem 4.2]{howard} or~\citet[Lemma XIV.1.3]{lang1}) then implies that $r_0$ can be chosen to be $r$ and $\mathcal{C}$ is a full class.

Concerning completeness, note that
\begin{multline}\label{eq:Levygu}
 \kappa(g_{u,\eta}(x))= \frac{1}{\eta} \int_0^{\eta} \Phi(s,u)e^{\langle \psi(s,u), x\rangle }\Bigg(\langle \beta, \psi(s,u)\rangle +\frac{1}{2}\langle \psi(s,u) \gamma\psi(s,u)\rangle\\
+ \int_V \left(e^{\langle \psi(s,u), \xi\rangle }-1-\langle \psi(s,u),\chi(\xi) \rangle\right)F(d\xi)\Bigg)ds.
\end{multline} 
Indeed, by Remark~\ref{rem:fullcomplete} (i), the integral
\[
\int_V \int_0^{\eta} \left| \Phi(s,u)e^{\langle \psi(s,u), x\rangle }\left(e^{\langle \psi(s,u), \xi\rangle }-1-\langle \psi(s,u),\chi(\xi) \rangle\right)\right| ds \,F(d\xi)
\]
is well-defined, whence by Fubini's theorem we can interchange the integration.
From~\eqref{eq:Levygu} it thus follows that
\begin{multline}
\lim_{\eta \to 0} \kappa(g_{u,\eta}(x))=\kappa(e^{\langle u, x\rangle})\\=e^{\langle u, x\rangle }\Bigg(\langle \beta, u\rangle +\frac{1}{2}\langle u, \gamma u\rangle
+ \int_V \left(e^{\langle u, \xi\rangle }-1-\langle u,\chi(\xi) \rangle\right)F(d\xi)\Bigg).\label{eq:Levyu}
\end{multline}
Moreover, by~\citet[Lemma II.2.44]{jacod} or simply as a consequence of the completeness of the class ${C}^{\ast}$, 
as defined in~\eqref{eq:CiV}, the function $u \mapsto \kappa(e^{\langle u, x\rangle})$ admits a unique representation of form~\eqref{eq:Levyu}, that is, if $\kappa(e^{\langle \cdot, x\rangle})$ also satisfies~\eqref{eq:Levyu} with
$(\widetilde{\beta}, \widetilde{\gamma},\widetilde{F})$, then $\beta = \widetilde{\beta}, \,\gamma = \widetilde{\gamma}$ and $F = \widetilde{F}$. This property carries over to the class $\mathcal{C}$. Indeed, for every $x \in D$, there exists some $\eta >0$ such that $\beta = \widetilde{\beta}, \,\gamma = \widetilde{\gamma}$ and $F = \widetilde{F}$ if $u \mapsto \kappa(g_{u,\eta}(x))$ also satisfies~\eqref{eq:Levygu} with $(\widetilde{\beta}, \widetilde{\gamma}, \widetilde{F})$. This proves that $\mathcal{C}$ is a complete class.
\end{proof}

In order to establish the semimartingale property of $X$ and to study its characteristics, 
we need to handle explosions and killing. Similar to~\citet{filipoyor}, we consider again the stopping times $T_{\Delta}$ defined in~\eqref{eq:lifetimeX} and $T_k'$ given by
\[
T_k':=\inf\{t \, |\, \| X_{t-}\|\geq k \textrm{ or } \| X_{t}\|\geq k\}, \quad k \geq 1.
\]
By the convention $\|\Delta\| =\infty$, $T_k' \leq T_{\Delta}$ for all $k \geq 1$. 
As a transition to $\Delta$ occurs either by a jump or by explosion, we additionally define the stopping times:
\begin{equation}\label{eq:stoppingtimes}
\begin{split}
\Tjump&=\left\{
\begin{array}{rl}
T_{\Delta}, & \textrm{if } T_k'=T_{\Delta} \textrm{ for some } k,\\
\infty, & \textrm{if } T_k' < T_{\Delta} \textrm{ for all } k, 
\end{array}\right.\\
\Texpl&=\left\{
\begin{array}{rl}
T_{\Delta}, & \textrm{if } T_k' < T_{\Delta} \textrm{ for all } k,\\
\infty, &\textrm{if } T_k'=T_{\Delta} \textrm{ for some } k, 
\end{array}\right.\\
T_k&=\left\{
\begin{array}{rl}
T_k', & \textrm{if } T_k' < T_{\Delta},\\
\infty, &\textrm{if } T_k'=T_{\Delta}. 
\end{array}\right.
\end{split}
\end{equation}

Note that $\{\Tjump < \infty\} \cap \{\Texpl < \infty\} =\emptyset$ and 
$\lim_{k \to \infty} T_k=\Texpl$ with $T_k < \Texpl$ on $\{\Texpl <\infty\}$. 
Hence $\Texpl$ is predictable with announcing sequence $T_k \wedge k$.
In order to turn $X$ into a semimartingale and to get explicit expressions for the characteristics, we stop $X$ before it explodes,
which is possible, since $\Texpl$ is predictable.
Note that we cannot stop $X$ before it is killed, as $\Tjump$ is totally inaccessible.
For this reason we shall concentrate on the process $(X_t^{\tau}):=(X_{t \wedge {\tau}})$, 
where $\tau$ is a stopping time satisfying $0 <\tau < \Texpl$, which exists by the above argument and the c\`adl\`ag property of $X$.
Since $X=X^{T_{\Delta}}$, we have 
\begin{align*}
X_t^{\tau}&=X_t1_{\{t < (\tau \wedge T_{\Delta})\}}+X_{\tau \wedge T_{\Delta}}1_{\{t \geq (\tau \wedge T_{\Delta})\}}\\
&=X_t1_{\{t < (\tau \wedge \Tjump)\}}+X_{\tau \wedge \Tjump}1_{\{t \geq (\tau \wedge \Tjump)\}},
\end{align*}
which implies that a transition to $\Delta$ can only occur through a jump.

Recall that $\Delta$ is assumed to be an arbitrary point which does not lie in $D$. 
We can thus identify $\Delta$ with some point in $V \setminus D$ 
such that every $C^2_b(D)$-function $f$ can be extended continuously 
to $D_{\Delta}$ with $f(\Delta)=0$.
Indeed, without loss of generality we may assume that such a point exists, because otherwise we can always embed $D_{\Delta}$ in $V \times \re$.

\begin{theorem}\label{th:semimartingale}
Let $X$ be an affine process and let $\tau$ be a stopping time with $\tau < \Texpl$, where $\Texpl$ 
is defined in~\eqref{eq:stoppingtimes}.
Then $X1_{[0,T_{\Delta})}$ and $X^{\tau}$ are semimartingales with state space $D \cup \{0\}$ and $D_{\Delta}$, respectively. Moreover, let $(B,C, \nu)$ denote the characteristics of $X^{\tau}$ relative to some truncation function $\chi$. Then 
there exists a version of $(B,C, \nu)$, which is of the form
\begin{equation}\label{eq:characteristicsfinal}
\begin{split}
B_{t,i}&=\int_0^{t\wedge \tau} b_i(X_{s-})ds,\\
C_{t,ij}&=\int_0^{t\wedge \tau} c_{ij}(X_{s-})ds,\\
\nu(\omega; dt,d\xi)&= K(X_t,d\xi)1_{[0,\tau]}dt,
\end{split}
\end{equation}
where $b: D \to V$ and $c: D \to S_+(V)$ are measurable functions 
and $K(x,d\xi)$ is a positive kernel from $(D,\mathcal{D})$ into $(V,\mathcal{B}(V))$, 
which satisfies $\int_{V} (\|\xi\|^2 \wedge 1)K(x,d\xi)< \infty$, 
$K(x,\{0\})=0$ and $x + \supp(K(x,\cdot)) \subseteq D_{\Delta}$ for all $x \in D$.
\end{theorem}

\begin{proof}
We adapt the proof of~\citet[Theorem 7.9 (ii), (iii)]{cinlar} to our setting.
By Lemma~\ref{lem:gmart}, 
\[
g_{u,\eta}(X)=\frac{1}{\eta}\int_0^{\eta}\Phi(s,u)e^{\langle \psi(s,u), X\rangle} ds
\] 
is a semimartingale for every $u \in \mathcal{U}$ and $\eta >0$. Since Lemma~\ref{lem:fullcomplete} asserts that $\mathcal{C}$, as defined in~\eqref{eq:mathcalC}, is a full class, an application of It\^o's formula to the function $h_i$ appearing in~\eqref{eq:fcth} shows that $X_i$ coincides with a semimartingale on each stochastic interval $[0, \tau_r[$, where 
\[
\tau_r=\inf\{t\geq 0 \, |\, \|X_t\|\geq r\}\wedge T_{\Delta}.
\]
Since we have $\mathbb{P}_x$-a.s.~$\lim_{r \to \infty} \tau_r =T_{\Delta}$ and since being a semimartingale is a local property (see~\citet[Proposition I.4.25]{jacod}), we conclude that $X1_{[0, T_{\Delta})}$ is a semimartingale.

Let now $\tau$ denote a stopping time with $\tau < \Texpl$. Then $X^{\tau}$ is also a semimartingale with state space $D_{\Delta}$, 
since explosion is avoided and the transition to $\Delta$ can only occur via killing, that is, a jump to $\Delta$, which is incorporated in the jump characteristic (see~\citet[Section 3]{filipoyor}). 

By~\citet[Theorem 6.25]{cinlar}, one can find a version of the characteristics $(B,C,\nu)$ of $X^{\tau}$, which is of the form
\begin{equation}\label{eq:characteristics}
\begin{split}
B_{t,i}&=\int_0^{t\wedge \tau} \widetilde{b}_{s-,i}dF_s,\\
C_{t,ij}&=\int_0^{t \wedge \tau} \widetilde{c}_{s-,ij}dF_s,\\
\nu(\omega; dt,d\xi)&=1_{[0,\tau]}dF_t(\omega)\widetilde{K}_{\omega,t} (d\xi),
\end{split}
\end{equation}
where $F$ is an additive process of finite variation, which is $\mathbb{P}_x$-indistinguishable from an $(\mathcal{F}_t)$-predictable process, $\widetilde{b}$ and $\widetilde{c}$ are $(\mathcal{F}_t)$-optional processes 
with values in $V$ and $S_+(V)$, respectively, and 
$\widetilde{K}_{\omega,t}(d\xi)$ is a positive kernel 
from $(\Omega \times \re_+, \mathcal{O}(\mathcal{F}_t))$\footnote{Here, $\mathcal{O}(\mathcal{F}_t)$ denotes the $(\mathcal{F}_t)$-optional $\sigma$-algebra.} 
into $(V, \mathcal{B}(V))$, which satisfies $\int_{V} (\|\xi\|^2 \wedge 1)\widetilde{K}_{\omega,t}(d\xi)< \infty$, 
$\widetilde{K}_{\omega,t}(\{0\})=0$ and $X_t(\omega) + \supp(\widetilde{K}_{\omega,t}) \subseteq D_{\Delta}$ for all $t \in [0, \tau]$ and $\mathbb{P}_x$-almost all $\omega$. 
Moreover, by~\citet[Theorem II.2.42]{jacod}, for every $f \in C^2_b(V)$, the process
\begin{multline}\label{eq:localmart}
f(X_t^{\tau})-f(x)-\int_0^{t\wedge \tau} \langle \widetilde{b}_{s-},\nabla f(X_{s-}) \rangle dF_s -
\frac{1}{2}\int_0^{t\wedge \tau}\sum_{i,j} \widetilde{c}_{s-,ij} D_{ij}f(X_{s-}) dF_s\\
-\int_0^{t\wedge \tau}\int_V \left(f(X_{s-}+\xi)-f(X_{s-})-\langle \nabla f(X_{s-}),\chi(\xi) \rangle\right)\widetilde{K}_{\omega,s-}( d\xi) dF_s
\end{multline}
is a $(\mathcal{F}_t,\mathbb{P}_x)$-local martingale and the last three terms are of finite variation. 
Let us denote 
\begin{multline*}
\widetilde{\mathcal{L}}f(X_{t-}(\omega)):=\langle \widetilde{b}_{t-},\nabla f(X_{t-}(\omega)) \rangle -\frac{1}{2}\sum_{i,j} \widetilde{c}_{t-,ij} D_{ij}f(X_{t-}(\omega)) \\
 -\int_V \left(f(X_{t-}(\omega)+\xi)-f(X_{t-}(\omega))-\langle \nabla f(X_{t-}(\omega)),\chi(\xi) \rangle\right)\widetilde{K}_{\omega, t-}( d\xi).
\end{multline*}

As proved in Lemma~\ref{lem:fullcomplete}, the class of functions $\mathcal{C}$ defined in~\eqref{eq:mathcalC} is complete. Let now $\mathcal{\widetilde{C}}\subset \mathcal{C}$ be the countable set satisfying the property stated in Definition~\ref{def:completeclass} and 
let $g_{\eta, u} \in \mathcal{\widetilde{C}}$ for some $u \in \im V$ and $\eta >0$.
Then Lemma~\ref{lem:gmart} and Definition~\ref{def:extendedgen} imply that
\begin{multline}
g_{\eta,u}(X_t^{\tau})-g_{\eta,u}(x)-\int_0^{t\wedge \tau} \mathcal{G}g_{\eta,u}(X_{s-}) ds \\
=g_{\eta,u}(X_t^{\tau})-g_{\eta,u}(x)-\int_0^{t\wedge \tau}\frac{1}{\eta}\left(P_{\eta}e^{\langle u, X_{s-}\rangle}-e^{\langle u,X_{s-}\rangle}\right)ds\label{eq:getamart}
\end{multline}
is a $(\mathcal{F}_t,\mathbb{P}_x)$-martingale, while $(\int_0^{t\wedge \tau} \mathcal{G}g_{\eta,u}(X_{s-}) ds)$ 
is a predictable finite variation process. Due to~\eqref{eq:localmart} and uniqueness of the canonical decomposition of the special semimartingale $g_{\eta,u}(X^{\tau})$ (see~\citet[Definition I.4.22, Corollary I.3.16]{jacod}), 
we thus have
\begin{align}\label{eq:eqgen}
\int_0^{t\wedge \tau}\widetilde{\mathcal{L}}g_{\eta,u}(X_{s-}) dF_s=\int_0^{t \wedge \tau} \mathcal{G}g_{\eta,u}(X_{s-}) ds \quad \textrm{ up to an evanescent set. }
\end{align}
Set now 
\[
\Lambda=\left\{(\omega,t): \widetilde{\mathcal{L}}g_{\eta,u}(X_{(t\wedge \tau \wedge T_{\Delta})-}(\omega))=0 \textrm{ for every } g_{\eta,u} \in \mathcal{\widetilde{C}}\right\}.
\]
Then the characteristic property~\eqref{eq:completeclass} of $\mathcal{\widetilde{C}}$ implies that $\Lambda$ is exactly the set where $\widetilde{b}=0$, $\widetilde{c}=0$ and $\widetilde{K}=0$. 
Hence we may replace $F$ by $1_{\Lambda^c}F$ without altering~\eqref{eq:characteristics}, that is, 
we can suppose that $1_{\Lambda}F=0$. This property together with~\eqref{eq:eqgen} implies that $dF_t \ll dt$ $\mathbb{P}_x$-a.s. 
Hence we know that there exists a triplet $(b', c', K')$ such that $F$ replaced by $t$ and $(\widetilde{b}, \widetilde{c}, \widetilde{K})$ replaced by $(b', c', K')$ satisfy all the conditions of~\eqref{eq:characteristics}.
In particular, we have by~\citet[Proposition II.2.9 (i)]{jacod} that $X^{\tau}$ is quasi-left continuous. Due to~\citet[Theorem 6.27]{cinlar}, it thus follows that 
\begin{align*}
b'_{t }&=b(X_t)1_{[0, \tau]},\\
c'_{t }&= c(X_t)1_{[0, \tau]},\\
K'_{\omega,t} (d\xi)&= K(X_t,d\xi)1_{[0, \tau]},
\end{align*}
where the functions $b$, $c$ and the kernel $K$ have the properties stated in~\eqref{eq:characteristicsfinal}.
This proves the assertion.
\end{proof}

\section{Regularity}\label{sec:regularity}

By means of the above derived semimartingale property, 
in particular the fact that the characteristics are absolutely continuous with respect 
to the Lebesgue measure, we can prove that every affine process is regular in the following sense:

\begin{definition}[Regularity]\label{def:regularity}
An affine process $X$ is called \emph{regular} if for every $u \in \mathcal{U}$ the derivatives
\begin{align}\label{eq:FRdef}
F(u) = \frac{\partial \Phi(t,u)}{\partial t}\Bigg |_{t=0}, \qquad
R(u) = \frac{\partial \psi(t,u)}{\partial t}\Bigg |_{t=0}
\end{align}
exist and are continuous on $\mathcal{U}^m$ for every $m \geq 1$.
\end{definition}

\begin{remark}
In the case of the canonical state space $D=\re_+^m \times \re^{n-m}$, the derivative of $\phi(t,u)$ at $t=0$ is also denoted by $F(u)$ (see~\citet[Equation (3.10)]{dfs} and Remark~\ref{rem:definitionaffine}). Since $\Phi(t,u)=e^{\phi(t,u)}$, we have
\[
\partial_t \Phi(t,u)|_{t=0}=e^{\phi(0,u)}\partial_t\phi(t,u)|_{t=0}=\partial_t \phi(t,u)|_{t=0}.
\]
Hence our definition of $F$ coincides with the one in~\citet{dfs}.
\end{remark}

\begin{lemma}\label{lem:finitevariation}
Let $X$ be an affine process. Then the functions $t \mapsto \Phi(t,u)$ and $t \mapsto \psi_i(t,u)$, $i \in \{1,\ldots,n\}$, defined in~\eqref{eq:affineprocess} are of finite variation for all $u \in \mathcal{U}$. 
\end{lemma}

\begin{proof}
Due to Assumption~\ref{ass:statespace}, there exist $n+1$ vectors such that $(x_1, \ldots, x_{n})$ are linearly independent and 
$x_{n+1}=\sum_{i=1}^n \lambda_ix_i$ for some $\lambda \in V$ with $\sum_{i=1}^n \lambda_i \neq 1$.

Let us now take $n+1$ affine processes $X^1, \ldots, X^{n+1}$ such that 
\[
 \mathbb{P}_{x_i}[X_0^i=x_i]=1
\]
for all $i \in \{1, \ldots,n+1\}$. It then follows from Theorem~\ref{th:semimartingale} that, for every $i \in \{1, \ldots,n+1\}$, $X^i$ is a semimartingale with respect to
the filtered probability space $(\Omega,\mathcal{F}, (\mathcal{F}_t), \mathbb{P}_{x_i})$. 
We can then construct a filtered probability space 
$(\Omega',\mathcal{F}', (\mathcal{F}'_t), \mathbb{P}')$, with respect to which $X_1, \ldots, X_{n+1}$ 
are independent semimartingales such that $\mathbb{P}' \circ (X^i)^{-1}= \mathbb{P}_{x_i}$.
One possible construction is the product probability space
$(\Omega^{n+1},\otimes_{i=1}^{n+1}\mathcal{F}, (\otimes_{i=1}^{n+1}\mathcal{F}_t),  \otimes_{i=1}^{n+1}\mathbb{P}_{x_i})$.

We write $y_i=(1, x_i)^{\top}$ and $Y^{i}=(1, X^{i})^{\top}$
for $i \in \{1,\ldots, n+1\}$. Then the definition of $x_i$ implies that $(y_1, \ldots, y_{n+1})$ are linearly independent.
Moreover, as $X^i$ exhibits c\`adl\`ag paths for all $i \in \{1, \ldots,n+1\}$, there exists some stopping time $\delta >0$ 
such that, for all $\omega \in \Omega'$ and $t \in [0, \delta(\omega))$, the vectors $(Y_{t}^{1}(\omega), \ldots, Y_{t}^{{n+1}}(\omega))$ are also linearly independent.
Let now $T >0$ and $u \in \mathcal{U}$ be fixed and choose some $0 < \varepsilon(\omega) \leq \delta(\omega)$ 
such that, for all $t \in [0, \varepsilon(\omega))$, $\Phi(T-t,u) \neq 0$.

Denoting the $(\mathcal{F}'_t, \mathbb{P}')$-martingales $\Phi(T-t,u)e^{\langle \psi(T-t,u), X^{i}_t\rangle}$ by $M_t^{T,u,i}$ 
and choosing the right branch of the complex logarithm, we thus have for all $t \in [0,\varepsilon(\omega))$
\[
\small
\left(\begin{array}{cccc}
1 &X_{t,1}^{1}(\omega) &\ldots&X_{t,n}^{1}(\omega)\\
\vdots &\vdots &\ddots& \vdots\\
1 &X_{t,1}^{n+1}(\omega) &\ldots&X_{t,n}^{n+1}(\omega)
\end{array}\right)^{-1}
\left(\begin{array}{c}
\ln M_t^{T,u,1}(\omega)\\
\vdots\\
\ln M_t^{T,u,{n+1}}(\omega)
\end{array}\right)=
\left(\begin{array}{c}
\ln \Phi(T-t,u)\\
\psi_1(T-t,u)\\
\vdots\\
\psi_n(T-t,u)
\end{array}\right).
\]
This implies that $(\Phi(s,u))_{s}$ and $(\psi(s,u))_{s}$ coincide on the stochastic 
interval $(T-\varepsilon(\omega),T]$ with deterministic semimartingales and are thus of finite variation. 
As this holds true for all $T>0$, we conclude that $t \mapsto \Phi(t,u)$ and $t \mapsto \psi_i(t,u)$ are of finite variation.
\end{proof}

Using Lemma~\ref{lem:finitevariation} and Theorem~\ref{th:semimartingale}, we are now prepared to prove regularity of affine processes.
Additionally, our proof reveals that the functions $F$ and $R$ 
have parameterizations of L\'evy-Khintchine type and that the (differential) 
semimartingale characteristics introduced in~\eqref{eq:characteristicsfinal}
depend in an affine way on $X$.

\begin{theorem}\label{th:regularity}
Every affine process is regular. Moreover, the functions $F$ and $R$, as defined in~\eqref{eq:FRdef}, are of the form
\begin{align*}
F(u) &=\left\langle u, b\right\rangle +\frac{1}{2}\left\langle u, a u\right\rangle-c \\
&\quad + \int_V \left(e^{\langle u, \xi \rangle}-1-\left\langle u, \chi(\xi)\right\rangle\right) m(d\xi), \quad u \in \mathcal{U},\\
\langle R(u),x\rangle &=\langle u,B(x)\rangle +\frac{1}{2}\left\langle u, A(x) u\right\rangle-\langle\gamma,x\rangle \\
&\quad + \int_V \left(e^{\langle u, \xi \rangle}-1-\left\langle u, \chi(\xi)\right\rangle\right) M(x,d\xi),  \quad u \in \mathcal{U},
\end{align*} 
where $\chi:V \to V$ denotes some truncation function such that $\chi(\Delta-x)=0$ for all $x \in D$, $b\in V$, $a \in S(V)$, $m$ is a (signed) measure, $c \in \re$, $\gamma \in V$ and $x \mapsto B(x)$, $x \mapsto A(x)$, $x \mapsto M(x,d\xi)$ are restrictions of linear maps on $V$ such that 
\begin{align*}
b(x)&=b+B(x),\\
c(x)&=a+A(x),\\
K(x,d\xi)&=m(d\xi)+M(x,d\xi)+(c +\langle \gamma, x\rangle)\delta_{(\Delta-x)}(d\xi).
\end{align*}
Here, the left hand side corresponds to the (differential) semimartingale characteristics introduced in~\eqref{eq:characteristicsfinal}.

Furthermore, on the set $\mathcal{Q}=\{(t,u)\in \re_+ \times \mathcal{U} \,|\, \Phi(s,u)\neq 0, \textrm{ for all } s \in [0,t]\}$, 
the functions $\Phi$ and $\psi$ satisfy the ordinary differential equations
\begin{align}
\partial_t \Phi(t,u)&=\Phi(t,u)F(\psi(t,u)), &\quad &\Phi(0,u)=1,\label{eq:RiccatiFgen}\\
\partial_t \psi(t,u)&=R(\psi(t,u)), &\quad &\psi(0,u)=u \in \mathcal{U}.\label{eq:RiccatiRgen}
\end{align}
\end{theorem}

\begin{remark}
Recall that without loss of generality we identify $\Delta$ with some point in $V \setminus D$ 
such that every $f \in C^2_b(D)$ can be extended continuously 
to $D_{\Delta}$ with $f(\Delta)=0$.
\end{remark}

\begin{proof}
Let $m \geq 1$ and $u \in \mathcal{U}^m$ be fixed and choose $T_u >0$ such that $\Phi(T_u-t,u) \neq 0$ for all $t \in [0,T_u]$. 
As $t \mapsto \Phi(t,u)$ and $t \mapsto \psi(t,u)$ are of finite variation by Lemma~\ref{lem:finitevariation}, 
their derivatives with respect to $t$ exist almost everywhere and
we can write
\begin{align*}
\Phi(T_u-t,u)-\Phi(T_u,u)&=\int_0^t -d\Phi(T_u-s,u),\\ 
\psi_i(T_u-t,u)-\psi_i(T_u,u)&=\int_0^t -d\psi_i(T_u-s,u),
\end{align*}
for $i \in \{1, \ldots,n\}$. Moreover, by the semiflow property of $\Phi$ and $\psi$ 
(see Proposition~\ref{prop:PhiPsiproperties}~\ref{item:semiflowPhiPsi}), 
differentiability of $\Phi(t,u)$ and $\psi(t,u)$ with respect to $t$ at some $\varepsilon \in (0, T_u]$ implies that the derivatives 
$ \partial_t|_{t=0}\psi(t,\psi(\varepsilon,u))$ and $\partial_t|_{t=0} \Phi(t,\psi(\varepsilon,u))$
exist as well. Let now $(\varepsilon_k)_{k \in \mathbb{N}}$ 
denote a sequence of points where $\Phi(t,u)$ and $\psi(t,u)$ are differentiable such that $\lim_{k \to \infty} \varepsilon_k =0$.
Then there exists a sequence $(u_k)_{k \in \mathbb{N}}$ given by
\begin{align}\label{eq:uk}
u_k=\psi(\varepsilon_k,u) \in \mathcal{U} \textrm{  with } \lim_{k \to \infty} u_k=u
\end{align}
such that the derivatives
\begin{align}\label{eq:diffev}
\partial_t|_{t=0} \psi(t,u_k), \quad \partial_t|_{t=0} \Phi(t,u_k)
\end{align}
exist for every $k \in \mathbb{N}$. Moreover, since $|\mathbb{E}_x[\exp(\langle u, X_{\varepsilon_k}\rangle)]|<m$, there exists some constant
$M$ such that $u_k \in \mathcal{U}^M$ for all $k \in \mathbb{N}$.

Furthermore, due to Theorem~\ref{th:semimartingale}, 
the canonical semimartingale representation of $X^{\tau}$ (see~\citet[Theorem II.2.34]{jacod}), where $\tau$ is a stopping time with $\tau < \Texpl$, is given by
\[
X_t^{\tau}=x+\int_0^{t\wedge \tau} b(X_{s-})ds +N_t^{\tau}+\int_0^{t\wedge \tau} \int_V (\xi- \chi(\xi))\mu^{X^{\tau}}(\omega; ds,d\xi),
\]
where $\mu^{X^{\tau}}$ is the random measure associated with the jumps of $X^{\tau}$ and
$N^{\tau}$ is a local martingale, namely the sum of the continuous martingale part and the purely discontinuous one, that is, 
\[
\int_0^{t\wedge \tau} \int_V \chi(\xi)(\mu^{X^{\tau}}(\omega; ds,d\xi)-K(X_{s-},d\xi) ds).
\]

Let now $(u_k)$ be given by~\eqref{eq:uk}.
Applying It\^o's formula (relative to the measure $\mathbb{P}_x$) to each of the martingales 
$M_{t\wedge \tau}^{T_{u_k},u_k}$= $\Phi(T_{u_k}-(t\wedge \tau),u)e^{\langle \psi(T_{u_k}-(t\wedge \tau),u_k), X_{t\wedge \tau}\rangle}$, $k \in \mathbb{N}$, we obtain 
\begin{align*}
&M_{t \wedge \tau}^{T_{u_k},u_k}\\
&=M_{0}^{T_{u_k},u_k}+\int_0^{t\wedge\tau} M_{s-}^{T_{u_k},u_k}\left(\frac{-d\Phi(T_{u_k}-s,u_k)}{\Phi(T_{u_k}-s,u_k)}+\left\langle -d\psi(T_{u_k}-s,u_k), X_{s-}\right\rangle\right)\\
&\quad +\int_0^{t\wedge\tau} M_{s-}^{T_{u_k},u_k}\Bigg(\left\langle \psi(T_{u_k}-s,u_k), b(X_{s-})\right\rangle\\ 
&\quad +\frac{1}{2}\left\langle \psi(T_{u_k}-s,u_k), c(X_{s-})\psi(T_{u_k}-s,u_k)\right\rangle\\
&\quad + \int_V \left(e^{\langle \psi(T_{u_k}-s,u_k), \xi \rangle}-1-\left\langle \psi(T_{u_k}-s,u_k), \chi(\xi)\right\rangle\right) K(X_{s-},d\xi)\Bigg) ds\\
&\quad + \int_0^{t\wedge\tau} M_{s-}^{T_{u_k},u_k} \left\langle \psi(T_{u_k}-s,u_k), dN^{\tau}_{s}\right\rangle\\
&\quad + \int_0^{t\wedge\tau} \int_V  M_{s-}^{T_{u_k},u_k}\left(e^{\langle \psi(T_{u_k}-s,u_k), \xi \rangle}-1-\left\langle \psi(T_{u_k}-s,u_k), \chi(\xi)\right\rangle\right)\\
&\quad \quad \quad \quad \quad  \times \left(\mu^{X^{\tau}}(\omega; ds,d\xi)-K(X_{s-},d\xi) ds\right).
\end{align*}
As the last two terms are local martingales and as the rest is of finite variation, 
we thus have, for almost all $t \in [0,T_{u_k} \wedge \tau]$, $\mathbb{P}_x$-a.s.~for every $x \in D$,
\begin{equation}
\begin{split}\label{eq:keyequality}
&\frac{d\Phi(T_{u_k}-t,u_k)}{\Phi(T_{u_k}-t,u_k)}+\left\langle d\psi(T_{u_k}-t,u_k), X_{t-}\right\rangle\\
&\quad=\left\langle \psi(T_{u_k}-t,u_k), b(X_{t-})\right\rangle dt  +\frac{1}{2}\left\langle \psi(T_{u_k}-t,u_k), c(X_{t-})\psi(T_{u_k}-t,u_k)\right\rangle dt\\
&\quad \quad + \int_V \left(e^{\langle \psi(T_{u_k}-t,u_k), \xi \rangle}-1-\left\langle \psi(T_{u_k}-t,u_k), \chi(\xi)\right\rangle\right) K(X_{t-},d\xi) dt.
\end{split}
\end{equation}
By setting $t=T_{u_k}$ on a set of positive measure with $\mathbb{P}_x[\tau \geq T_{u_k}]$ and letting $T_{u_k} \to 0$, we obtain due to~\eqref{eq:diffev} for each $k \in \mathbb{N}$ and $x \in D$
\begin{equation}
\begin{split}
&\partial_t|_{t=0} \Phi(t,u_k)+\left\langle \partial_t|_{t=0}\psi(t,u_k), x\right\rangle\\
&\quad=\left\langle u_k, b(x)\right\rangle dt  +\frac{1}{2}\left\langle u_k, c(x)u_k\right\rangle dt
+ \int_V \left(e^{\langle u_k, \xi \rangle}-1-\left\langle u_k, \chi(\xi)\right\rangle\right) K(x,d\xi) dt.
\end{split}
\end{equation}
Since the right hand side is continuous in $u_k$, which is a consequence of the 
support properties of $K(x,\cdot)$ and the fact that $u_k \in \mathcal{U}^M$ for all $k \in \mathbb{N}$, the limit for $u_k \to u$ of the left hand side exists as well. By the affine independence of the $n+1$ elements in $D$, the coefficients $\partial_t|_{t=0} \Phi(t,u_k)$ and  $\partial_t|_{t=0}\psi(t,u_k)$ converge
for $u_k \to u$, whence the limit is affine, too.
Since $m \geq 1$ and $u$ was arbitrary, it follows that 
\[
 \left\langle u, b(x)\right\rangle dt  +\frac{1}{2}\left\langle u, c(x)u\right\rangle dt
+ \int_V \left(e^{\langle u, \xi \rangle}-1-\left\langle u, \chi(\xi)\right\rangle\right) K(x,d\xi) dt
\]
is an affine function in $x$ for all $u \in \mathcal{U}$.
 
By uniqueness of the L\'evy-Khintchine representation and the assumption that $D$ contains $n+1$ affinely independent elements, 
this implies that $x \mapsto b(x)$, $x \mapsto c(x)$ and $x \mapsto K(x,d\xi)$ are affine functions in the following sense:
\begin{equation*}
\begin{split}
b(x)&=b+B(x),\\
c(x)&=a+A(x),\\
K(x,d\xi)&=m(d\xi)+M(x,d\xi)+(c +\langle \gamma, x\rangle)\delta_{(\Delta-x)}(d\xi),
\end{split}
\end{equation*}
where $b\in V$, $a \in S(V)$, $m$ a (signed) measure, $c \in \re$, $\gamma \in V$ and $x \mapsto B(x)$, $x \mapsto A(x)$, $x \mapsto M(x,d\xi)$ 
are restrictions of linear maps on $V$.
Indeed, $c +\langle \gamma, x\rangle$ corresponds to the killing rate of the process, which is incorporated in the jump measure.
Here, we explicitly use the convention that $e^{\langle u, \Delta \rangle}=0$, $b(\Delta)=0$, $c(\Delta)=0$, $K(\Delta,d\xi)=0$ and the fact that $\chi(\Delta-x) =0$ for all
$x \in D$.

Moreover, for $t$ small enough, we have for all $u \in \mathcal{U}$
\begin{align*}
\Phi(t,u)-\Phi(0,u)&=\int_0^t \Phi(s,u) \Bigg(\left\langle \psi(s,u), b\right\rangle +\frac{1}{2}\left\langle \psi(s,u), a\psi(s,u)\right\rangle-c \\
&\quad + \int_V \left(e^{\langle \psi(s,u), \xi \rangle}-1-\left\langle \psi(s,u), \chi(\xi)\right\rangle\right) m(d\xi)\Bigg) ds,\\
\langle\psi(t,u)-\psi(0,u),x\rangle&=\int_0^t \Bigg(\langle\psi(s,u),B(x)\rangle  +\frac{1}{2}\left\langle \psi(s,u), A(x)\psi(s,u)\right\rangle-\langle\gamma,x\rangle \\&\quad + \int_V \left(e^{\langle \psi(s,u), \xi \rangle}-1-\left\langle \psi(s,u), \chi(\xi)\right\rangle\right) M(x,d\xi)\Bigg) ds.
\end{align*}
Note again that the properties of the support of $K(x, \cdot)$ carry over to the measures $M(x, \cdot)$ and $m(\cdot)$ implying that the above integrals are well-defined.
Due to the continuity of $t \mapsto \Phi(t,u)$ and $t \mapsto \psi(t,u)$, we can conclude that the derivatives of $\Phi$ and $\psi$ 
exist at $0$ and are continuous on $\mathcal{U}^m$ for every $m \geq 1$, since they are given by
\begin{align*}
F(u) = \frac{\partial \Phi(t,u)}{\partial t}\Bigg |_{t=0}&=\left\langle u, b\right\rangle +\frac{1}{2}\left\langle u, a u\right\rangle-c \\
&\quad + \int_V \left(e^{\langle u, \xi \rangle}-1-\left\langle u, \chi(\xi)\right\rangle\right) m(d\xi),\\
\langle R(u),x\rangle = \left \langle \frac{\partial \psi(t,u)}{\partial t}\Bigg |_{t=0}, x \right\rangle&=\langle u,B(x)\rangle +\frac{1}{2}\left\langle u, A(x) u\right\rangle-\langle\gamma,x\rangle \\
&\quad + \int_V \left(e^{\langle u, \xi \rangle}-1-\left\langle u, \chi(\xi)\right\rangle\right) M(x,d\xi).
\end{align*}
This proves the first part of the theorem.

By the regularity of $X$, we are now allowed to differentiate the semiflow equations~\eqref{eq:flowprop} on the set $\mathcal{Q}=\{(t,u)\in \re_+ \times \mathcal{U} \,|\, \Phi(s,u)\neq 0, \textrm{ for all } s\in [0,t]\}$ with respect
to $s$ and evaluate them at $s=0$. As a consequence, $\Phi$ and
$\psi$ satisfy~\eqref{eq:RiccatiFgen} and~\eqref{eq:RiccatiRgen}.
\end{proof}

\begin{remark}
The differential equations~\eqref{eq:RiccatiFgen} and~\eqref{eq:RiccatiRgen} 
are called \emph{generalized Riccati equations}, which is due to the particular form 
of $F$ and $R$.
\end{remark}

\begin{remark}\label{rem:compactstatespace}
Using the results of Theorem~\ref{th:regularity}, in particular the assertion on the semimartingale characteristics, we aim to construct examples of affine processes on compact state spaces, which justify that we do not restrict ourselves to unbounded sets $D$. For simplicity, let $n=1$. Then the pure deterministic drift process with characteristics
\[
b(x)=b+Bx, \quad B \leq 0,\, -Br_1 \leq b \leq -Br_2, \quad c(x)=0,  \quad K(x,d\xi)=0
\]
is an affine process on the interval $[r_1,r_2]$. Another example of an affine process on a compact, but discrete, state space of the form
\[
 \{0, 1, \ldots, k\}
\]
can be obtained by a pure jump process $X$ with jump size distribution $\delta_{1}(d\xi)$ and intensity $k-X$. In terms of the semimartingale characteristics, we thus have
$b(x)=0$, $c(x)=0$ and $K(x,d\xi)=(k-x)\delta_1(d\xi)$. 
For such type of jump processes the state space is necessarily discrete and cannot be extended to the whole interval $[0,k]$. In the presence of a diffusion component, the state space is necessarily unbounded, since the stochastic invariance conditions on $c(x)=a+A x$, which would guarantee that the process remains 
in some interval $[r_1,r_2]$ imply
\[
 a+Ar_1=0 \quad \textrm{and} \quad a+Ar_2=0,
\]
yielding $a=A = 0$. 
\end{remark}


\end{document}